\documentclass[11pt]{article}
 \usepackage{epsfig}
 \usepackage{amssymb,amsmath,amsthm,amscd}
\usepackage{latexsym}
\usepackage{showlabels}
\usepackage{graphicx}
\usepackage{color}
\usepackage{subfigure}
 \theoremstyle{theorem}
 \newtheorem{thm}{Theorem}[section]
 \newtheorem{lemma}[thm]{Lemma}
  
 \newtheorem{definition}[thm]{Definition}
 \newtheorem{prop}[thm]{Proposition}
 \newtheorem{cor}[thm]{Corollary}
  \newtheorem{remark}[thm]{Remark}
 \title{Necessary conditions for existence of tensor invariants for general nonlinear dynamical systems}
\author{
Zitong Zhao\footnote{College of Mathematics, Jilin University, Changchun, Jilin 130012, China ({\tt  zhaozt22@mails.jlu.edu.cn}).},   \;
Shaoyun Shi\footnote{School of Mathematics and Statistics, Changchun University of Science and Technology, Changchun 130022, China ({\tt  shisy@jlu.edu.cn}).},
\;
Wenlei Li\footnote{College of Mathematics, Jilin University,
Changchun, Jilin 130012,  China ({\tt  lwlei@jlu.edu.cn}).},
\;
Zhiguo Xu\footnote{College of Mathematics, Jilin University,
Changchun, Jilin 130012, China ({\tt xuzg2014@jlu.edu.cn}).}
\;
and Kaiyin Huang\footnote{School of Mathematics, Sichuan University, Chengdu 630065,  China ({\tt huangky@scu.edu.cn}).}
}
\date{\empty}

\begin{document}
 \maketitle

 \begin{abstract}
The integrability has been playing an essential role in the field of differential equations. This property may better help us obtain the topological structure and even the global dynamics for the considered system. A system is called integrable if it has a number of tensor invariants, which can comprehensively define the integrability problem. In this paper, we give necessary conditions for existence of tensor invariants for general nonlinear systems, especially semi-quasihomogeneous systems. Our results may be viewed as a generalization of Poincar\'e and Kozlov's work.

 \end{abstract}
\noindent
{\bf Keywords:} Integrability; Tensor invariant; Quasi-homogeneous system; Kovalevskaya exponents.
 \newpage

\section{Introduction}\label{sec1}

One fundamental problem in the field of differential equations is to detect whether a given system is integrable or not.
A system is called integrable if it has a number of invariants such that it can be solved by quadrature or in a closed form. The integrability property of a system may help us obtain the topological structure and even the global dynamics \cite{ref13,xx24}. On the contrary, non-integrability of a system may always push us to expect that the system admits complex dynamical behaviors or chaotic phenomena \cite{xx22,xx23}.

As far as we know, the study of integrability may be traced back originally to Poincar\'e at the end of the 19th century. He gave a criterion on first integrals by using eigenvalues of the Jacobian matrix of the vector field at the fixed point, i.e., if the eigenvalues $\lambda_1,\cdots,\lambda_n$ do not satisfy any resonant equality of the following type
$$
\sum_{j=1}^n k_j\lambda_j=0, ~~k_j\in\mathbb{N}, ~~\sum_{j=1}^n k_j\geq 1,
$$
then the system does not have any nontrivial integral analytic in a neighborhood of the origin \cite{xx25}. Following this idea, Furta \cite{ref2}, Shi et al. \cite{ref23,xx35}, Zhang et al. \cite{xx18,xx3} and Li et al. \cite{xx27} further presented comprehensive results from different perspectives.

Besides above results around the fixed point, which can be regarded as local integrability, there are also a significant number of lecturers investigating issues of integrability in the neighborhood of non-trivial solutions. Just like the case of the fixed point, this perspective may also be traced back to Poincar\'e. He proposed some necessary conditions for the integrability of Hamiltonian systems along the periodic orbits by using Floquet's theory. In 1982, Ziglin extended it to complex analytic Hamiltonian systems. He presented a necessary condition for complete integrability around the Riemann solutions by using single-valued group theory \cite{xx33}. In 1990s, based on Poincar\'e and Ziglin's theory, Morales-Ruiz and Ramis gave the results to complex analytic Hamiltonian systems by using differential Galois theory \cite{xx36}. For more relevant results, see in Yoshida \cite{xx34}, Simon \cite{xx29}, Maciejewski, Przybylska\cite{xx39,xx40}, and Li \cite{xx32,xx26} et al. Meanwhile, there are some other methods around the non-trivial solutions, such as Lilbre, Zhang et al. by using invariant algebraic curves and integrating factors \cite{xx31,xx37,xx38} etc.

Based on the development of complex analysis, people have also always discussed the integrability property of differential equations. For the nonlinear systems, the relevant ideas may be traced back to Kovalevslaya considering all integrable cases of the gyroscopic problem by analyzing the single-valued property of solutions \cite{xx28}. Inspired by the above, people developed a set of singularity analysis methods for the integrability of differential equations, such as Painlev\'e \cite{xx43}. In 1983, Yoshida gave a necessary condition on weighted homogeneous first integrals $\Phi(x^1,\cdots,x^n)$ of weighted degree $M$ for the similarity invariant system. Suppose, the elements of the vector
$$
{\rm grad~}\Phi(c)=[\frac{\partial\Phi}{\partial x^1}(c_1,\cdots,c_n),\cdots, \frac{\partial\Phi}{\partial x^n}(c_1,\cdots,c_n)]
$$
are finite and not identically zero for a fixed choice, then $\rho=M$ becomes a Kovalevskaya exponent. In 1996, Goriely extended Yoshida's method to study partial integrability as follows \cite{ref13}.
Later Furta and Shi et al. also presented some relevant results \cite{ref23}.

With above abundant ideas, people have presented many important results for different types of integrability problems. Consider an analytic system of differential equations
\begin{equation}\label{int}
\dot{x}=F(x),\  x=(x^{1},\cdots,x^{n})\in\mathbb{C}^{n}.
\end{equation}
Several interesting kinds of definitions of integrability can be concluded as follows.

$\bullet$ Completely integrability \cite{xx2,xx3,ref13}: $n-1$ functionally independent first integrals.

$\bullet$ Lie integrability \cite{xx17,xx16}: $n$ linearly independent and commuting symmetries (including vector field $F$ itself).

$\bullet$ Bogoyavlenskij integrability \cite{m8,ref1,xx5}: $k$ functionally independent first integrals $\Phi_{1},\cdots, \Phi_{k}$ and $n-k$ linearly independent vector fields $v_{1}=F(x),\cdots,v_{n-k}$ such that
$$
[v_{i},v_{j}]=0,\ v_{i}(\Phi_{l})=0,\ 1\leq i,j\leq n-k, \ 1\leq l\leq k.
$$

$\bullet$ Jacobi integrability \cite{xx19}: $n-2$ functionally independent first integrals $\Phi_{1},\cdots, \Phi_{k}$ and a Jacobian multiplier.

$\bullet$ Euler-Jacobi-Lie integrability \cite{xx19}: $k$ functionally independent first integrals, $n-k-1$ independent symmetry fields $v_{1}= F,\cdots,v_{n-k}$ generating a nilpotent Lie algebra of the vector fields, and an invariant volume n-form $\Omega$ such that $L_{v_{i}}\Omega=0,\ L_{v_{i}}\Phi_{j}=0,1\leq i\leq n-k-1,\ 1\leq j\leq k$.

It is worth noting that these notions can be uniformly described as tensor invariants. In fact, recall a smooth tensor field $T$ of type $(p,q)$
\begin{equation}\label{TF}
T=T_{j_{1}\cdots j_{q}}^{i_{1}\cdots i_{p}}(x)\frac{\partial}{\partial x^{i_{1}}}\otimes\cdots\otimes\frac{\partial}{\partial x^{i_{p}}}\otimes dx^{j_{1}}\otimes\cdots\otimes dx^{j_{q}},~~i_{r}, j_{s}\in\{1,2,\cdots,n\},
\end{equation}
where components $T^{i_{1}\cdots i_{p}}_{j_{1}\cdots j_{q}}(x)$ are smooth functions, and let $L_{F}$ be the Lie derivative along the vector field $F=(F^1,\cdots,F^n)$, i.e.,
\begin{align}
 (L_{F}T)^{i_{1}\cdots i_{p}}_{j_{1}\cdots j_{q}} = & F^{s}\frac{\partial}{\partial x^{s}}T^{i_{1}\cdots i_{p}}_{j_{1}\cdots j_{q}}+T^{i_{1}\cdots i_{p}}_{kj_{2}\cdots j_{q}}\frac{\partial F^{k}}{\partial x^{j_{1}}}+\cdots+T^{i_{1}\cdots i_{p}}_{j_{1}\cdots j_{q-1}k}\frac{\partial F^{k}}{\partial x^{j_{q}}}\notag
 \\& -T^{li_{2}\cdots i_{p}}_{j_{1}\cdots j_{q}}\frac{\partial F^{i_{1}}}{\partial x^{l}}-\cdots-T^{i_{1}\cdots i_{p-1}l}_{j_{1}\cdots j_{q}}\frac{\partial F^{i_{p}}}{\partial x^{l}}.\label{eq2}
 \end{align}
Here and throughout we use the Einstein summation convention. Specially, $T$ is called an analytic tensor field if all the components of the tensor field are analytic.
\begin{definition}
A tensor field $T$ is called a tensor invariant of system (\ref{int}), if $L_{F}T=0$, or equivalently
$(L_{F}T)^{i_{1}\cdots i_{p}}_{j_{1}\cdots j_{q}}=0$
for any $i_{r}, j_{s}\in\{1,2,\cdots,n\}$. \end{definition}

Following the notions above, one can easily find first integrals are tensor invariants of type $(0,0)$, Lie symmetries are tensor invariants of type $(1,0)$ and invariant $n$-form is also a tensor invariant of type $(0,n)$.

In 1992, Kozlov extended Yoshida's method to a necessary condition for existence of  general  tensor invariants which are non-zero at the balance of the quasi-homogeneous system \cite{xx1}. Let system has a quasi-homogeneous tensor invariant $T$ of degree $m$ and let $T(c)\neq 0$. Then there can be foung integers $1\leq i_1,\cdots,i_p, j_1,\cdots, j_q \leq n$ such that
$$
\rho_{i_1}+\cdots+\rho_{i_p}-\rho_{j_1}-\rho_{j_q}+m=0.
$$
For tensor invariants of type $(1,0)$, Furta's method was extended to Lie symmetries \cite{xx16}. If the quasi-homogeneous system possesses an analytic Lie symmetry of degree $M\geq 0$, then for some $i\in\{0,\cdots,n\}$, a resonant relation holds as follows
$$
\sum_{j=1}^n k_j\lambda_j=l\lambda_i, ~~k_j\in\mathbb{N}, ~~\sum_{j=1}^n k_1\geq M.
$$

In this paper, we try to present a general and completed investigation on the existence of tensor invariants in the general nonlinear differential equations. First of all, we generalize Poincar\'e's method to tensor invariants in a neighborhood of the fixed point. Based on this result we consider the problem in the semi-quasihomogeneous systems by making the variational equation in a neighborhood of the scale-invariant solution and present a necessary condition on the existence of the tensor invariants in the semi-quasihomogeneous systems. It may be noted that as mentioned above, Kozlov ever gave the similar result under some restrictions. However our result may be viewed as a generalization without his restrictions.

This paper is organized as follows. In Section 2, we present a necessary condition for the existence of tensor invariants in the general differential equations. In Section 3, we present a necessary condition for the existence of tensor invariants in the semi-quasihomogeneous systems. In Section 4, we give two examples to illustrate our theorems.

\section{Tensor invariants for general nonlinear systems}
\subsection{Trivial tensor invariant}

In this subsection, we introduce the concept of trivial tensor invariant and present some characterization.
\begin{definition}
If $L_{F}T=0$ holds for any vector field $F$, then $T$ is called a trivial tensor invariant.
\end{definition}

Obviously, $T=c$ is the trivial tensor invariant of type $(0,0)$ (trivial first integral), in what follows we try to provide a complete characterization of the trivial tensor invariants. Let $S_k(i_1,\cdots,i_k)$ denote the set of all permutations of the indices $(i_1,\cdots,i_k)$, and for simplicity, denote $S_k(1,\cdots,k)$ as $S_k$.
\begin{prop}\label{trivial}
If the tensor field $T$ given by \eqref{TF} is trivial, then all of components $T_{j_{1}\cdots j_{q}}^{i_{1}\cdots i_{p}}(x)$ are constant and $p=q$.
Moreover, $T_{j_{1}\cdots j_{p}}^{i_{1}\cdots i_{p}}=0$ for any $(j_1,\cdots,j_p)\notin \Theta_p$, where $\Theta_p= \{(j_1,\cdots,j_p)|(j_1,\cdots,j_p)=\sigma(i_1,\cdots,i_p), \sigma\in S_p(i_1,\cdots,i_p)\}$.
\end{prop}
\begin{proof}
According to definition of trivial tensor invariant,
\begin{equation*}
 (L_{F}T)^{i_{1}\cdots i_{p}}_{j_{1}\cdots j_{q}}=0, ~~i_{r}, j_{s}\in\{1,2,\cdots,n\}
\end{equation*}
hold for every analytic vector filed $F$.

First of all, by taking $F$ to be $(0,\cdots,1,\cdots, 0)$(with $1$ in the $h-$th position), one can conclude that for any components $T_{j_{1}\cdots j_{q}}^{i_{1}\cdots i_{p}}(x)$,
$$
\frac{\partial}{\partial x^{h}}T^{i_{1}\cdots i_{p}}_{j_{1}\cdots j_{q}}=0,~~h=1,2,\cdots,n,
$$
which means $T_{j_{1}\cdots j_{q}}^{i_{1}\cdots i_{p}}(x)$ are constant.

Next, by choosing $F=(0,\cdots,x_h,\cdots, 0)$, we obtain
\begin{align}\label{pq}
T_{j_{1}\cdots j_{q}}^{i_{1}\cdots i_{p}}(\sum_{k=1}^q \delta_{h}^{j_k}-\sum_{l=1}^p \delta_{h}^{i_l})=0, ~~h=1,2,\cdots,n,
\end{align}
where $\delta^{i}_{j}$ is the Kronecker symbol. Thus
$$
\sum\limits_{k=1}^q \delta_{h}^{j_k}-\sum\limits_{l=1}^p \delta_{h}^{i_l}=0, ~~h=1,2,\cdots,n
$$
hold for any $T_{j_{1}\cdots j_{q}}^{i_{1}\cdots i_{p}}\neq 0$, which implies that $p=q$.


Finally, if $(j_1,\cdots,j_p)\notin \Theta_p$, then $\sum\limits_{k=1}^q \delta_{h}^{j_k}-\sum\limits_{l=1}^p \delta_{h}^{i_l}\ne 0$, consequently $T_{j_{1}\cdots j_{q}}^{i_{1}\cdots i_{p}}=0$ according to \eqref{pq}.
\end{proof}

According to Proposition \ref{trivial}, the trivial tensor invariant has the form
\begin{align}\label{TTI}
T=\sum_{\sigma\in S_p(i_{1},\cdots, i_{p})}\sum_{i_1,\cdots,i_p=1}^nT_{\sigma(i_{1}\cdots i_{p})}^{i_{1}\cdots i_{p}}\frac{\partial}{\partial x^{i_{1}}}\otimes\cdots\otimes\frac{\partial}{\partial x^{i_{p}}}\otimes dx^{i_{\sigma(1)}}\otimes\cdots\otimes dx^{i_{\sigma(p)}}.
\end{align}
where $(x_{i_{\sigma_{1}}}, \cdots, x_{i_{\sigma_{p}}})=\sigma(i_1,\cdots,i_p)$.

We now proceed to investigate the relationships among $T_{\sigma(i_{1}\cdots i_{p})}^{i_{1}\cdots i_{p}}$.

\begin{prop}
 The trivial tensor invariant of type $(1,1)$ has the form
\begin{align}\label{t1-1}
T(x)=c\sum_{i=1}^n\frac{\partial}{\partial x^i}\otimes dx^i,
\end{align}
where $c$ is a constant.
\end{prop}
\begin{proof}
By \eqref{TTI}, the trivial tensor invariant of type $(1,1)$ has the form
$$
T(x)=\sum_{i=1}^n \beta_i \frac{\partial}{\partial x^i}\otimes dx^i.
$$
Let $F=(x^j, 0, \cdots, 0)$, then
$$
(L_FT)_j^1=\beta_1\frac{\partial F^1}{\partial x^j}-\beta_j\frac{\partial F^1}{\partial x^j}=\beta_1-\beta_j=0, ~~j\in \{1,\cdots,n\},
$$
so $T(x)$ takes the form \eqref{t1-1}.
\end{proof}
\begin{prop}
The trivial tensor invariant of type $(2,2)$ has the form
\begin{align}\label{t2-2}
T(x)=\sum_{\sigma\in S_2}c_{\sigma}\sum_{i_1,i_2=1}^n\frac{\partial}{\partial x^{i_1}}\otimes\frac{\partial}{\partial x^{i_2}}\otimes dx^{i_{\sigma(1)}}\otimes dx^{i_{\sigma(2)}},
\end{align}
where $c_{\sigma}$ is a constant.
\end{prop}
\begin{proof}
By \eqref{TTI}, the trivial tensor invariant of type $(2,2)$ has the form
$$
T(x)=\sum_{\sigma\in S_2(i_{1},i_{2})}\sum_{i_1,i_2=1}^n T^{i_1i_2}_{\sigma(i_1i_2)} \frac{\partial}{\partial x^{i_1}}\otimes \frac{\partial}{\partial x^{i_2}}\otimes dx^{i_{\sigma(1)}}\otimes dx^{i_{\sigma(2)}}.
$$

We claim that for any $i_1,i_2,i\in\{1,\cdots,n\}$ with $i_1\neq i_2$,
\begin{align}\label{T12}
T^{i_1i_2}_{\sigma(i_1i_2)}=T^{12}_{\sigma(12)}
\end{align}
and
\begin{align}\label{T11}
T^{ii}_{ii}=\sum\limits_{\sigma\in S_2} T^{12}_{\sigma(12)}.
\end{align}

In fact, if $n=2$, by choosing vector field $F$ with $F^{1}=x^{2}, F^{2}=x^{1}$, we obtain
\begin{align}
& (L_FT)^{11}_{12}=T^{11}_{11}\frac{\partial F^{1}}{\partial x^{2}}-T^{21}_{12}\frac{\partial F^{1}}{\partial x^{2}}-T^{12}_{12}\frac{\partial F^{1}}{\partial x^{2}}=T^{11}_{11}-T^{21}_{12}-T^{12}_{12}=0,\label{01112}\\
& (L_FT)^{21}_{11}=T^{21}_{21}\frac{\partial F^{2}}{\partial x^{1}}+T^{21}_{12}\frac{\partial F^{2}}{\partial x^{1}}-T^{11}_{11}\frac{\partial F^{2}}{\partial x^{1}}=T^{21}_{21}+T^{21}_{12}-T^{11}_{11}=0,\label{02111}\\
& (L_FT)^{12}_{11}=T^{12}_{21}\frac{\partial F^{2}}{\partial x^{1}}+T^{12}_{12}\frac{\partial F^{2}}{\partial x^{1}}-T^{11}_{11}\frac{\partial F^{2}}{\partial x^{1}}=T^{12}_{21}+T^{12}_{12}-T^{11}_{11}=0\label{01211}.
\end{align}
By \eqref{01112}, \eqref{02111} and \eqref{01211}, we get $T^{21}_{\sigma(21)}=T^{12}_{\sigma(12)}$, i.e., \eqref{T12} hold.

If $n\geq 3$, then for any $i_3\in\{1,\cdots,n\}$ with $i_3\ne i_1, i_3\ne i_2$,
we choose vector field $F$ with $F^{i_1}=x^{i_3}, F^{i_2}=x^{i_3}$, then based on the fact that $T_{j_{1} j_{2}}^{i_{1}i_{2}}=0$ when $(j_1,j_2)\notin \Theta_2$,
\begin{align*}
& (L_FT)^{i_{1}i_{2}}_{\sigma(i_{1}i_{3})}=T^{i_{1}i_{2}}_{\sigma(i_{1}i_{2})}\frac{\partial F^{i_2}}{\partial x^{i_3}}-T^{i_{1}i_{3}}_{\sigma(i_{1}i_{3})}\frac{\partial F^{i_2}}{\partial x^{i_3}}=T^{i_{1}i_{2}}_{\sigma(i_{1}i_{2})}-T^{i_{1}i_{3}}_{\sigma(i_{1}i_{3})}=0,\\
& (L_FT)^{i_{1}i_{2}}_{\sigma(i_{3}i_{2})}=T^{i_{1}i_{2}}_{\sigma(i _{1}i_{2})}\frac{\partial F^{i_1}}{\partial x^{i_3}}-T^{i_{3}i_{2}}_{\sigma(i_{3}i_{2})}\frac{\partial F^{i_1}}{\partial x^{i_3}}=T^{i_{1}i_{2}}_{\sigma(i _{1}i_{2})}-T^{i_{3}i_{2}}_{\sigma(i_{3}i_{2})}=0,
\end{align*}
which lead to
\begin{align*}
T^{i_{1}i_{3}}_{\sigma(i_{1}i_{3})}=
T^{i_{1}i_{2}}_{\sigma(i _{1}i_{2})}=T^{i_{3}i_{2}}_{\sigma(i_{3}i_{2})}.
\end{align*}
Owing to the arbitrary choice of $i_1,i_2,i_3$, \eqref{T12} hold.

To prove \eqref{T11}, we choose $j\in\{1,\cdots,n\}$ with $j\neq i$ and $F$ with $F^{i}=x^{j}$, then
\begin{align*}
& (L_FT)^{ii}_{ij}=T^{ii}_{ii}\frac{\partial F^{i}}{\partial x^{j}}-T^{ji}_{ij}\frac{\partial F^{i}}{\partial x^{j}}-T^{ij}_{ij}\frac{\partial F^{i}}{\partial x^{j}}=T^{ii}_{ii}-T^{ji}_{ij}-T^{ij}_{ij}=0,
\end{align*}
which together with \eqref{T12} imply
\begin{align*}
T^{ii}_{ii}=T^{ij}_{ij}+T^{ji}_{ij}=\sum\limits_{\sigma\in S_2}T^{12}_{\sigma(12)}.
\end{align*}

Setting $c_{\sigma}=T^{12}_{\sigma(12)}$, it follows from \eqref{T12} and \eqref{T11} that \eqref{t2-2} holds.
\end{proof}

\begin{prop}
The trivial tensor invariant of type $(3,3)$ has the form
\begin{align}\label{t3-3}
T(x)=\sum_{\sigma\in S_3}c_{\sigma}\sum_{i_1,i_2,i_3=1}^n\frac{\partial}{\partial x^{i_1}}\otimes\frac{\partial}{\partial x^{i_2}}\otimes\frac{\partial}{\partial x^{i_3}}\otimes dx^{i_{\sigma(1)}}\otimes dx^{i_{\sigma(2)}}\otimes dx^{i_{\sigma(3)}},
\end{align}
where $c_{\sigma}$ is a constant.
\end{prop}
\begin{proof}
We claim that for any $i_1,i_2,i_3,i\in\{1,\cdots,n\}$ with $i_1, i_2, i_3$ are pairwise distinct,
\begin{align}\label{T123}
T^{i_1i_2i_3}_{\sigma(i_1i_2i_3)}=T^{123}_{\sigma(123)},
\end{align}
\begin{align}\label{T111}
T^{iii}_{iii}=\sum\limits_{\sigma\in S_3} T^{123}_{\sigma(123)}
\end{align}
and
\begin{align}\label{T112}
& T^{i_1i_1i_2}_{i_1i_1i_2}=c_{\sigma_1}+c_{\sigma_3}, ~~T^{i_1i_1i_2}_{i_1i_2i_1}=c_{\sigma_2}+c_{\sigma_4}, ~~T^{i_1i_1i_2}_{i_2i_1i_1}=c_{\sigma_5}+c_{\sigma_6},\notag\\
& T^{i_1i_2i_1}_{i_1i_2i_1}=c_{\sigma_1}+c_{\sigma_6}, ~~T^{i_1i_2i_1}_{i_1i_1i_2}=c_{\sigma_2}+c_{\sigma_5}, ~~T^{i_1i_2i_1}_{i_2i_1i_1}=c_{\sigma_3}+c_{\sigma_4},\notag\\
& T^{i_2i_1i_1}_{i_2i_1i_1}=c_{\sigma_1}+c_{\sigma_2}, ~~T^{i_2i_1i_1}_{i_1i_2i_1}=c_{\sigma_3}+c_{\sigma_5}, ~~T^{i_2i_1i_1}_{i_1i_1i_2}=c_{\sigma_4}+c_{\sigma_6},
\end{align}
where $c_{\sigma_k}=T^{123}_{\sigma_k(123)}$ which is arranged in the lexicographical order.

In fact, if $n=3$, by choosing vector field $F$ with $F^{i_1}=x^{i_3}, F^{i_3}=x^{i_1}$, we obtain
\begin{align}
& (L_FT)^{i_{1}i_{1}i_{2}}_{i_{1}i_{3}i_{2}}=T^{i_{1}i_{1}i_{2}}_{i_{1}i_{1}i_{2}}\frac{\partial F^{i_1}}{\partial x_{i_3}}-T^{i_{3}i_{1}i_{2}}_{i_{1}i_{3}i_{2}}\frac{\partial F^{i_1}}{\partial x_{i_3}}-T^{i_{1}i_{3}i_{2}}_{i_{1}i_{3}i_{2}}\frac{\partial F^{i_1}}{\partial x_{i_3}}=T^{i_{1}i_{1}i_{2}}_{i_{1}i_{1}i_{2}}-T^{i_{3}i_{1}i_{2}}_{i_{1}i_{3}i_{2}}
-T^{i_{1}i_{3}i_{2}}_{i_{1}i_{3}i_{2}}=0,\label{112132}\\
& (L_FT)^{i_{3}i_{1}i_{2}}_{i_{1}i_{1}i_{2}}=T^{i_{3}i_{1}i_{2}}_{i_{3}i_{1}i_{2}}\frac{\partial F^{i_3}}{\partial x_{i_1}}+T^{i_{3}i_{1}i_{2}}_{i_{1}i_{3}i_{2}}\frac{\partial F^{i_3}}{\partial x_{i_1}}-T^{i_{1}i_{1}i_{2}}_{i_{1}i_{1}i_{2}}\frac{\partial F^{i_3}}{\partial x_{i_1}}=T^{i_{3}i_{1}i_{2}}_{i_{3}i_{1}i_{2}}
+T^{i_{3}i_{1}i_{2}}_{i_{1}i_{3}i_{2}}
-T^{i_{1}i_{1}i_{2}}_{i_{1}i_{1}i_{2}}=0,\label{312112}\\
& (L_FT)^{i_{3}i_{2}i_{1}}_{i_{1}i_{2}i_{1}}=T^{i_{3}i_{2}i_{1}}_{i_{3}i_{2}i_{1}}\frac{\partial F^{i_3}}{\partial x_{i_1}}+T^{i_{3}i_{2}i_{1}}_{i_{1}i_{2}i_{3}}\frac{\partial F^{i_3}}{\partial x_{i_1}}-T^{i_{1}i_{2}i_{1}}_{i_{1}i_{2}i_{1}}\frac{\partial F^{i_3}}{\partial x_{i_1}}=T^{i_{3}i_{2}i_{1}}_{i_{3}i_{2}i_{1}}
+T^{i_{3}i_{2}i_{1}}_{i_{1}i_{2}i_{3}}
-T^{i_{1}i_{2}i_{1}}_{i_{1}i_{2}i_{1}}=0,\label{321121}\\
& (L_FT)^{i_{1}i_{2}i_{1}}_{i_{1}i_{2}i_{3}}=T^{i_{1}i_{2}i_{1}}_{i_{1}i_{2}i_{1}}\frac{\partial F^{i_1}}{\partial x_{i_3}}-T^{i_{3}i_{2}i_{1}}_{i_{1}i_{2}i_{3}}\frac{\partial F^{i_1}}{\partial x_{i_3}}-T^{i_{1}i_{2}i_{3}}_{i_{1}i_{2}i_{3}}\frac{\partial F^{i_1}}{\partial x_{i_3}}=T^{i_{1}i_{2}i_{1}}_{i_{1}i_{2}i_{1}}
-T^{i_{3}i_{2}i_{1}}_{i_{1}i_{2}i_{3}}
-T^{i_{1}i_{2}i_{3}}_{i_{1}i_{2}i_{3}}=0\label{121123}.
\end{align}
By \eqref{112132}, \eqref{312112}, \eqref{321121} and \eqref{121123}, we get
\begin{align*}
T^{i_{1}i_{2}i_{3}}_{i_{1}i_{2}i_{3}}=T^{i_{3}i_{2}i_{1}}_{i_{3}i_{2}i_{1}},
~T^{i_{1}i_{3}i_{2}}_{i_{1}i_{3}i_{2}}=T^{i_{3}i_{1}i_{2}}_{i_{3}i_{1}i_{2}}.
\end{align*}
Owing to the arbitrary choice of $i_1,i_2,i_3$,
$T^{i_{1}i_{2}i_{3}}_{i_{1}i_{2}i_{3}}=T^{123}_{123}$.

By choosing vector field $F$ with $F^{i_1}=x^{i_2}, F^{i_2}=x^{i_1}, F^{i_3}=x^{i_1}$, we obtain
\begin{align}
& (L_FT)^{i_{1}i_{2}i_{3}}_{i_{2}i_{1}i_{1}}=T^{i_{1}i_{2}i_{3}}_{i_{2}i_{3}i_{1}}\frac{\partial F^{i_3}}{\partial x_{i_1}}+T^{i_{1}i_{2}i_{3}}_{i_{2}i_{1}i_{3}}\frac{\partial F^{i_3}}{\partial x_{i_1}}-T^{i_{1}i_{2}i_{1}}_{i_{2}i_{1}i_{1}}\frac{\partial F^{i_3}}{\partial x_{i_1}}=T^{i_{1}i_{2}i_{3}}_{i_{2}i_{3}i_{1}}
+T^{i_{1}i_{2}i_{3}}_{i_{2}i_{1}i_{3}}
-T^{i_{1}i_{2}i_{1}}_{i_{2}i_{1}i_{1}}=0,\label{123211}\\
& (L_FT)^{i_{1}i_{2}i_{1}}_{i_{2}i_{1}i_{3}}=T^{i_{1}i_{2}i_{1}}_{i_{2}i_{1}i_{1}}\frac{\partial F^{i_1}}{\partial x_{i_3}}-T^{i_{3}i_{2}i_{1}}_{i_{2}i_{1}i_{3}}\frac{\partial F^{i_1}}{\partial x_{i_3}}-T^{i_{1}i_{2}i_{3}}_{i_{2}i_{1}i_{3}}\frac{\partial F^{i_1}}{\partial x_{i_3}}=T^{i_{1}i_{2}i_{1}}_{i_{2}i_{1}i_{1}}
-T^{i_{3}i_{2}i_{1}}_{i_{2}i_{1}i_{3}}
-T^{i_{1}i_{2}i_{3}}_{i_{2}i_{1}i_{3}}=0,\label{121213}\\
& (L_FT)^{i_{1}i_{2}i_{3}}_{i_{1}i_{3}i_{1}}=T^{i_{1}i_{2}i_{3}}_{i_{2}i_{3}i_{1}}\frac{\partial F^{i_2}}{\partial x_{i_1}}+T^{i_{1}i_{2}i_{3}}_{i_{1}i_{3}i_{2}}\frac{\partial F^{i_2}}{\partial x_{i_1}}-T^{i_{1}i_{1}i_{3}}_{i_{1}i_{3}i_{1}}\frac{\partial F^{i_2}}{\partial x_{i_1}}=T^{i_{1}i_{2}i_{3}}_{i_{2}i_{3}i_{1}}
+T^{i_{1}i_{2}i_{3}}_{i_{1}i_{3}i_{2}}
-T^{i_{1}i_{1}i_{3}}_{i_{1}i_{3}i_{1}}=0,\label{123131}\\
& (L_FT)^{i_{1}i_{1}i_{3}}_{i_{1}i_{3}i_{2}}=T^{i_{1}i_{1}i_{3}}_{i_{1}i_{3}i_{1}}\frac{\partial F^{i_1}}{\partial x_{i_2}}-T^{i_{2}i_{1}i_{3}}_{i_{1}i_{3}i_{2}}\frac{\partial F^{i_1}}{\partial x_{i_2}}-T^{i_{1}i_{2}i_{3}}_{i_{1}i_{3}i_{2}}\frac{\partial F^{i_1}}{\partial x_{i_2}}=T^{i_{1}i_{1}i_{3}}_{i_{1}i_{3}i_{1}}
-T^{i_{2}i_{1}i_{3}}_{i_{1}i_{3}i_{2}}
-T^{i_{1}i_{2}i_{3}}_{i_{1}i_{3}i_{2}}=0\label{113132},
\end{align}
By \eqref{123211}, \eqref{121213}, \eqref{123131} and \eqref{113132}, we get
\begin{align*}
T^{i_{1}i_{2}i_{3}}_{i_{2}i_{3}i_{1}}=T^{i_{3}i_{2}i_{1}}_{i_{2}i_{1}i_{3}}
=T^{i_{2}i_{1}i_{3}}_{i_{1}i_{3}i_{2}}.
\end{align*}
Owing to the arbitrary choice of $i_1,i_2,i_3$,
$T^{i_{1}i_{2}i_{3}}_{i_{2}i_{3}i_{1}}=T^{123}_{231}$.

Similarly, \eqref{T123} hold.

If $n\geq 4$, then for any $i_4\in\{1,\cdots,n\}$ with $i_4\ne i_1, i_4\ne i_2, i_4\ne i_3$,
we choose vector field $F$ with $F^{i_1}=x^{i_4}, F^{i_2}=x^{i_4}, F^{i_3}=x^{i_4}$, then based on the fact that $T_{j_{1} j_{2}j_{3}}^{i_{1}i_{2}i_{3}}=0$ when $(j_1,j_2,j_3)\notin \Theta_3$,
\begin{align*}
&
(L_FT)^{i_{1}i_{2}i_{3}}_{\sigma(i_{1}i_{2}i_{4})}=T^{i_{1}i_{2}i_{3}}_{\sigma(i_{1}i_{2}i_{3})}\frac{\partial F^{i_3}}{\partial x^{i_4}}-T^{i_{1}i_{2}i_{4}}_{\sigma(i_{1}i_{2}i_{4})}\frac{\partial F^{i_3}}{\partial x^{i_4}}=T^{i_{1}i_{2}i_{3}}_{\sigma(i_{1}i_{2}i_{3})}
-T^{i_{1}i_{2}i_{4}}_{\sigma(i_{1}i_{2}i_{4})}=0,\\
& (L_FT)^{i_{1}i_{2}i_{3}}_{\sigma(i_{1}i_{4}i_{3})}=T^{i_{1}i_{2}i_{3}}_{\sigma(i_{1}i_{2}i_{3})}\frac{\partial F^{i_2}}{\partial x^{i_4}}-T^{i_{1}i_{4}i_{3}}_{\sigma(i_{1}i_{4}i_{3})}\frac{\partial F^{i_2}}{\partial x^{i_4}}=T^{i_{1}i_{2}i_{3}}_{\sigma(i_{1}i_{2}i_{3})}
-T^{i_{1}i_{4}i_{3}}_{\sigma(i_{1}i_{4}i_{3})}=0,\\
& (L_FT)^{i_{1}i_{2}i_{3}}_{\sigma(i_{4}i_{2}i_{3})}=T^{i_{1}i_{2}i_{3}}_{\sigma(i_{1}i_{2}i_{3})}\frac{\partial F^{i_1}}{\partial x^{i_4}}-T^{i_{4}i_{2}i_{3}}_{\sigma(i_{4}i_{2}i_{3})}\frac{\partial F^{i_1}}{\partial x^{i_4}}=T^{i_{1}i_{2}i_{3}}_{\sigma(i_{1}i_{2}i_{3})}
-T^{i_{4}i_{2}i_{3}}_{\sigma(i_{4}i_{2}i_{3})}=0,
\end{align*}
which leads to
\begin{align*}
T^{i_1i_2i_3}_{\sigma(i_{1}i_{2}i_{3})}
=T^{i_{1}i_{2}i_{4}}_{\sigma(i_{1}i_{2}i_{4})}
=T^{i_{1}i_{4}i_{3}}_{\sigma(i_{1}i_{4}i_{3})}
=T^{i_{4}i_{2}i_{3}}_{\sigma(i_{4}i_{2}i_{3})}.
\end{align*}
Owing to the arbitrary choice of $i_1,i_2,i_3,i_4$, \eqref{T123} hold.

To prove \eqref{T112}, we choose vector field $F$ with $F^{i_1}=x^{i_3}$, then
\begin{align*}
(L_FT)^{i_{1}i_{1}i_{2}}_{i_{1}i_{3}i_{2}}=T^{i_{1}i_{1}i_{2}}_{i_{1}i_{1}i_{2}}\frac{\partial F^{i_1}}{\partial x_{i_3}}-T^{i_{3}i_{1}i_{2}}_{i_{1}i_{3}i_{2}}\frac{\partial F^{i_1}}{\partial x_{i_3}}-T^{i_{1}i_{3}i_{2}}_{i_{1}i_{3}i_{2}}\frac{\partial F^{i_1}}{\partial x_{i_3}}=T^{i_{1}i_{1}i_{2}}_{i_{1}i_{1}i_{2}}-T^{i_{3}i_{1}i_{2}}_{i_{1}i_{3}i_{2}}
-T^{i_{1}i_{3}i_{2}}_{i_{1}i_{3}i_{2}}=0,
\end{align*}
Similarly, together with \eqref{T123}, \eqref{T112} hold.

To prove \eqref{T111}, we choose $j\in\{1,\cdots,n\}$ with $j\neq i$ and $F$ with $F^{j}=x^{i}$, then
\begin{align*}
(L_FT)_{iii}^{jii}&=T_{jii}^{jii}\frac{\partial F^{j}}{\partial x_{i}}+T_{iji}^{jii}\frac{\partial F^{j}}{\partial x_{i}}+T_{iij}^{jii}\frac{\partial F^{j}}{\partial x_{i}}-T^{iii}_{iii}\frac{\partial F^{j}}{\partial x_{i}}\\
&=T_{jii}^{jii}
+T_{iji}^{jii}
+T_{iij}^{jii}
-T^{iii}_{iii}=0,
\end{align*}
which together with \eqref{T112} \eqref{T123} imply
\begin{align*}
T^{iii}_{iii}=\sum\limits_{\sigma\in S_3}T^{123}_{\sigma(123)}.
\end{align*}

Setting $c_{\sigma}=T^{123}_{\sigma(123)}$, it follows from \eqref{T123}, \eqref{T111} and \eqref{T112} that \eqref{t3-3} holds.
\end{proof}

Consequently, we can give a conjecture that the trivial tensor invariant of type $(p,p)$ has the form
\begin{align*}
T(x)=\sum_{\sigma\in S_p} c_{\sigma}\sum_{i_1,\cdots,i_p=1}^n\frac{\partial}{\partial x^{i_1}}\otimes\cdots\otimes\frac{\partial}{\partial x^{i_p}}\otimes dx^{i_{\sigma(1)}}\otimes\cdots\otimes dx^{i_{\sigma(p)}},
\end{align*}
where $c_{\sigma}$ is a constant.

\subsection{Necessary condition for the existence of analytic tensor invariants}
Consider an analytic system of differential equations
\begin{equation}\label{eq01}
\dot{x}=F(x),\  x=(x^{1},\cdots,x^{n})\in\mathbb{C}^{n},
\end{equation}
where $F(x)=(F^{1}(x),\cdots,F^{n}(x))$ is an analytic vector-valued function and $F(0)=0$. System \eqref{eq01} can be rewritten as
\begin{equation}\label{eq1}
\dot{x}=Ax+\tilde{f}(x)
\end{equation}
in a neighbourhood of $x=0$, where $A$ is the Jacobi matrix of the vector field $F(x)$ at $x=0$, and $\tilde{f}(x)=O(x^2)$.

To prove the main theorem of this section, we need the following lemma.

\begin{lemma}\label{lemma1}
If system (\ref{eq1}) has an analytic tensor invariant in a neighbourhood of $x=0$ which can be written as
\begin{equation}\label{exp2}
T=T^{(k)}+T^{(k+1)}+\cdots,
\end{equation}
where $T^{(r)}(r=k,k+1,\cdots,)$ is an analytic tensor field of type $(p,q)$ whose components are homogeneous polynomials of degree $r$, and $T^{(k)}\not\equiv 0$,
then $T^{(k)}$ is an analytic tensor invariant of linear system
\begin{equation}\label{eq00}
\dot{x}=Ax.
\end{equation}
\end{lemma}
\begin{proof}
Since $T$ is a tensor invariant of system (\ref{eq1}),
\begin{equation}\label{eq02}
\begin{aligned}
 (L_{F}T)^{i_{1}\cdots i_{p}}_{j_{1}\cdots j_{q}}=0
 \end{aligned}
\end{equation}
holds for any $i_{r}, j_{s}\in\{1,2,\cdots,n\}$. Substituting (\ref{exp2}) into (\ref{eq02}) gives
\begin{align}\label{com1}
0=&[(Ax)^{s}+\tilde{f}^{s}(x)]\frac{\partial}{\partial x^{s}}(T^{(k)i_{1}\cdots i_{p}}_{j_{1}\cdots j_{q}}+T^{(k+1)i_{1}\cdots i_{p}}_{j_{1}\cdots j_{q}}+\cdots)\notag\\
&
+(T^{(k)i_{1}\cdots i_{p}}_{sj_{2}...j_{q}}+T^{(k+1)i_{1}\cdots i_{p}}_{sj_{2}...j_{q}}+\cdots)(\frac{\partial (Ax)^{s}}{\partial x^{j_{1}}}+\frac{\partial \tilde{f}^{s}}{\partial x^{j_{1}}})+\cdots\notag\\&
+(T^{(k)i_{1}\cdots i_{p}}_{j_{1}\cdots j_{q-1}s}+T^{(k+1)i_{1}\cdots i_{p}}_{j_{1}\cdots j_{q-1}s}+\cdots)(\frac{\partial (Ax)^{s}}{\partial x^{j_{q}}}+\frac{\partial \tilde{f}^{s}}{\partial x^{j_{q}}})\notag\\&
-(T^{(k)li_{2}\cdots i_{p}}_{j_{1}\cdots j_{q}}+T^{(k+1)li_{2}\cdots i_{p}}_{j_{1}\cdots j_{q}}+\cdots)(\frac{\partial (Ax)^{i_{1}}}{\partial x^{l}}+\frac{\partial \tilde{f}^{i_{1}}}{\partial x^{l}})-\cdots\notag\\&
-(T^{(k)i_{1}\cdots i_{p-1}l}_{j_{1}\cdots j_{q}}+T^{(k+1)i_{1}\cdots i_{p-1}l}_{j_{1}\cdots j_{q}}+\cdots)(\frac{\partial (Ax)^{i_{p}}}{\partial x^{l}}+\frac{\partial \tilde{f}^{i_{p}}}{\partial x^{l}}).
\end{align}
Now, comparing the coefficients of the lowest powers of $x$ in both sides of (\ref{com1}), we get
\begin{align*}
(Ax)^{s}\frac{\partial}{\partial x^{s}}T^{(k)i_{1}\cdots i_{p}}_{j_{1}\cdots j_{q}}+T^{(k)i_{1}\cdots i_{p}}_{sj_{2}\cdots j_{q}}\frac{\partial (Ax)^{s}}{\partial x^{j_{1}}}+\cdots+T^{(k)i_{1}\cdots i_{p}}_{j_{1}\cdots j_{q-1}s}\frac{\partial (Ax)^{s}}{\partial x^{j_{q}}}&\notag\\
-T^{(k)li_{2}\cdots i_{p}}_{j_{1}\cdots j_{q}}\frac{\partial (Ax)^{i_{1}}}{\partial x^{l}}-\cdots-T^{(k)i_{1}\cdots i_{p-1}l}_{j_{1}\cdots j_{q}}\frac{\partial (Ax)^{i_{p}}}{\partial x^{l}}&=0,
\end{align*}
which implies that  $T^{(k)}$ is a tensor invariant of linear system (\ref{eq00}). The lemma is proved.

\begin{cor}\label{corollary1}
Assume system (\ref{eq1}) is linear. If $T$ is an analytic tensor invariant of (\ref{eq1}), then $T^{(r)}(r=k,k+1,\cdots,)$ is an analytic tensor invariant of system (\ref{eq1}).
\end{cor}

\end{proof}
\begin{thm}\label{th1}
If system (\ref{eq1}) possesses an analytic tensor invariant of type $(p,q)$ of order $k$ in a neighbourhood of $x=0$, then at least one of the following resonant condition holds:
\begin{equation}\label{eq4}
\sum_{j=1}^{n}k_{j}\lambda_{j}
=\lambda_{i_{1}}+\cdots+\lambda_{i_{p}}-\lambda_{j_{1}}-\cdots-\lambda_{j_{q}}, ~~i_{r}, j_{s}\in\{1,2,\cdots,n\},
\end{equation}
where $\lambda_1, \cdots,\lambda_n$ are eigenvalues of $A$, $k_{j} \in \mathbb{N}, \sum\limits_{j=1}^{n}k_{j}=k$.
\end{thm}
\begin{proof}
Let $T$ be an analytic tensor invariant of type $(p,q)$ of order $k$ of system (\ref{eq1}), without loss of generality, we assume that $T$ has the form of \eqref{exp2}.
According to Lemma \ref{lemma1}, $T^{(k)}$ is an analytic tensor invariant of system (\ref{eq00}).

Since after a nonsingular linear transformation, $A$ can be changed to a Jordan canonical form, for simplicity, we assume that $A$ is a Jordan canonical form, i.e.,
\begin{equation}
A={\left(\begin{array}{cccc}
J_{1} &  &  &  \\
& J_{2} & & \\
 & &\ddots &  \\
 &  &  & J_{m}
\end{array}
\right)},~~~~
J_{r}={\left(\begin{array}{cccc}
\lambda_{r} & 1 & &   \\
 & \ddots & \ddots & \\
 & & \ddots  & 1 \\
 & &  & \lambda_{r}
\end{array}
\right)},
\end{equation}
where $J_{r}(r=1,\cdots,m)$ is a Jordan block with degree equal to $d_{r}, d_{1}+\cdots+d_{m}=n$, and $ \lambda_{1},\cdots,\lambda_{m}$ can be equal.

Make the following transformation of variables
\begin{equation}
x=Cv,
\end{equation}
where
\begin{equation}
C={\left(\begin{array}{cccc}
C_{1} &  &  &  \\
& C_{2} & & \\
 & &\ddots &  \\
 &  &  & C_{m}
\end{array}
\right)},
C_{r}={\left(\begin{array}{cccc}
1 &  & &   \\
 & \varepsilon &  & \\
 & & \ddots  &  \\
 & &  & \varepsilon^{d_{r}-1}
\end{array}
\right)},
\end{equation}
$\varepsilon > 0$ is a constant, then system (\ref{eq00}) can be rewritten as
\begin{equation}\label{eq03}
\dot{v}=(B+\varepsilon \tilde{B})v,
\end{equation}
where
\begin{equation}
B={\left(\begin{array}{cccc}
B_{1} &  &  &  \\
& B_{2} & & \\
 & &\ddots &  \\
 &  &  & B_{m}
\end{array}
\right)},
B_{r}={\left(\begin{array}{ccc}
\lambda_{r} &  &   \\
 &  \ddots  &  \\
 &  & \lambda_{r}
\end{array}
\right)},
\end{equation}
\begin{equation}
\tilde{B}={\left(\begin{array}{cccc}
\tilde{B}_{1} &  &  &  \\
& \tilde{B}_{2} & & \\
 & &\ddots &  \\
 &  &  & \tilde{B}_{m}
\end{array}
\right)},
\tilde{B}_{r}={\left(\begin{array}{cccc}
0 & 1 & &   \\
 & \ddots & \ddots & \\
 & & \ddots  &  1\\
 & &  & 0
\end{array}
\right)}.
\end{equation}
Obviously, $\tilde{T}^{(k)}(v,\varepsilon)=T^{(k)}(Cv)$ is an analytic tensor invariant of linear system (\ref{eq03}), that is,
\begin{equation}\label{eq04}
L_{H}\tilde{T}^{(k)}(v,\varepsilon)=0,
\end{equation}
where $H$ denotes the vector field of system (\ref{eq03}), i.e., $H(v)=(B+\varepsilon \tilde{B})v$.

Noticing that for any $i_{r}, j_{s}\in\{1,2,\cdots,n\}$,
\begin{equation}\nonumber
(\tilde{T}^{(k)})_{j_{1}\cdots j_{q}}^{i_{1}\cdots i_{p}}(v,\varepsilon)=(c_{j_{1}j_{1}}\cdots c_{j_{q}j_{q}}c^{-1}_{i_{1}i_{1}}\cdots c^{-1}_{i_{p}i_{p}})(T^{(k)})_{j_{1}\cdots j_{q}}^{i_{1}\cdots i_{p}}(Cv),
\end{equation}
where $c_{ij}$ are the elements of matrix $C$, $\tilde{T}^{(k)}(v,\varepsilon)$ has the form
\begin{equation}\label{eq004}
\tilde{T}^{(k)}(v,\varepsilon)=\varepsilon^{l}\tilde{T}^{(k)}_{l}(v)+\varepsilon^{l+1}\tilde{T}^{(k)}_{l+1}(v)\cdots+\varepsilon^{L}\tilde{T}^{(k)}_{L}(v),
\end{equation}
where $l$, $L$ are integers, $\tilde{T}^{(k)}_{r}(v)$ are tensor fields whose components are all homogeneous functions of degree $k$, and $\tilde{T}^{(k)}_{l}(v)\not\equiv 0$.
By (\ref{eq04}) and (\ref{eq004}), we have
\begin{align}\label{eqq004}
0=&((B+\varepsilon \tilde{B})v)^{s}\frac{\partial }{\partial v^{s}}(\varepsilon^{l}(\tilde{T}^{(k)}_{l})_{j_{1}\cdots j_{q}}^{i_{1}\cdots i_{p}}(v)+\cdots+\varepsilon^{L}(\tilde{T}^{(k)}_{L})_{j_{1}\cdots j_{q}}^{i_{1}\cdots i_{p}}(v))\notag
\\& +(\varepsilon^{l}(\tilde{T}^{(k)}_{l})_{sj_{2}\cdots j_{q}}^{i_{1}\cdots i_{p}}(v)+\cdots+\varepsilon^{L}(\tilde{T}^{(k)}_{L})_{sj_{2}\cdots j_{q}}^{i_{1}\cdots i_{p}}(v))\frac{\partial((B+\varepsilon\tilde{B})v)^{s}}{\partial v^{j_{1}}}+\cdots\notag
\\& +(\varepsilon^{l}(\tilde{T}^{(k)}_{l})_{j_{1}\cdots j_{q-1}s}^{i_{1}\cdots i_{p}}(v)+\cdots+\varepsilon^{L}(\tilde{T}^{(k)}_{L})_{j_{1}\cdots j_{q-1}s}^{i_{1}\cdots i_{p}}(v))\frac{\partial((B+\varepsilon\tilde{B})v)^{s}}{\partial v^{j_{q}}}\notag
\\& -(\varepsilon^{l}(\tilde{T}^{(k)}_{l})_{j_{1}\cdots j_{q}}^{si_{2}\cdots i_{p}}(v)+\cdots+\varepsilon^{L}(\tilde{T}^{(k)}_{L})_{j_{1}\cdots j_{q}}^{si_{2}\cdots i_{p}}(v))\frac{\partial((B+\varepsilon\tilde{B})v)^{i_{1}}}{\partial v^{s}}-\cdots\notag
\\& -(\varepsilon^{l}(\tilde{T}^{(k)}_{l})_{j_{1}\cdots j_{q}}^{i_{1}\cdots i_{p-1}s}(v)+\cdots+\varepsilon^{L}(\tilde{T}^{(k)}_{L})_{j_{1}\cdots j_{q}}^{i_{1}\cdots i_{p-1}s}(v))\frac{\partial((B+\varepsilon\tilde{B})v)^{i_{p}}}{\partial v^{s}}.
\end{align}
Since $\tilde{T}^{(k)}_{l}(v)\not\equiv 0$, there exists some $\hat{i}_{m}, \hat{j}_{n}\in\{1,2,\cdots,n\}$ such that $\psi(v)\triangleq(\tilde{T}^{(k)}_{l})_{\hat{j}_{1}\cdots \hat{j}_{q}}^{\hat{i}_{1}\cdots \hat{i}_{p}}(v)\neq 0$.
Now equating all the terms in (\ref{eqq004}) of the same order with respect to $\varepsilon$ to zero, we obtain
\begin{equation}\label{eq06}
(Bv)^{s}\frac{\partial }{\partial v^{s}}\psi(v)+\psi(v)\lambda_{\hat{j}_{1}}+\cdots+\psi(v)\lambda_{\hat{j}_{q}}-\psi(v)\lambda_{\hat{i}_{1}}-\cdots--\psi(v)\lambda_{\hat{i}_{p}}=0.
\end{equation}
Suppose
\begin{equation}\label{eq07}
\psi(v)=\sum_{k_{1}+\cdots+k_{n}=k}\psi_{k_{1}\cdots k_{n}}(v^{1})^{k_{1}}\cdots(v^{n})^{k_{n}}.
\end{equation}
Substituting (\ref{eq07}) into (\ref{eq06}), we get
\begin{equation}\label{eq10}
\begin{aligned}
& \sum_{k_{1}+\cdots+k_{n}=k}[(k_{1}+\cdots+k_{d_{1}})\lambda_{1}+\cdots+(k_{d_{1}+\cdots+d_{m-1}+1}k_{n})\lambda_{m})]\psi_{k_{1}\cdots k_{n}}(v^{1})^{k_{1}}\cdots(v^{n})^{k_{n}}
\\& +\lambda_{\hat{j}_{_{1}}}\psi_{k_{1}...k_{n}}(v^{1})^{k_{1}}\cdots(v^{n})^{k_{n}}+\cdots+\lambda_{\hat{j}_{q}}\psi_{k_{1}\cdots k_{n}}(v^{1})^{k_{1}}\cdots(v^{n})^{k_{n}}\\& -\lambda_{\hat{i}_{_{1}}}\psi_{k_{1}\cdots k_{n}}(v^{1})^{k_{1}}\cdots(v^{n})^{k_{n}}+\cdots-\lambda_{\hat{i}_{p}}\psi_{k_{1}\cdots k_{n}}(v^{1})^{k_{1}}\cdots(v^{n})^{k_{n}}=0.
\end{aligned}
\end{equation}
After the rearrangement, (\ref{eq10}) becomes
\begin{equation*}
\sum_{k_{1}+\cdots+k_{n}=k}(k_{1}\lambda_{1}+\cdots+k_{n}\lambda_{n}+\lambda_{\hat{j}_{1}}+\cdots+\lambda_{\hat{j}_{q}}
-\lambda_{\hat{i}_{1}}-\cdots-\lambda_{\hat{i}_{p}})\psi_{k_{1}\cdots k_{n}}(v^{1})^{k_{1}}\cdots(v^{n})^{k_{n}}=0.
\end{equation*}
 Therefore, a resonant condition of type (\ref{eq4}) has to be fulfilled for any nonzero coefficient $\psi_{k_{1}\cdots k_{n}}$.
\end{proof}
\begin{remark}
(1) If det $A=0$, then at least one resonant relation of type (\ref{eq4}) holds. In fact, if $\lambda_{1}=0$, we only need to take $k_{1}=k, k_{2}=\cdots=k_{n}=0$ and $i_1=\cdots=i_p=j_1=\cdots=j_q=1$.

(2) When $0\leq q< p$, (\ref{eq4}) possesses the solutions independent of values of $\lambda_{1},\cdots,\lambda_{n}$. In fact, any element in the set
\begin{equation*}
\Sigma=\{(k_{1},\cdots, k_{n},i_{1},\cdots,i_{p},j_{1},\cdots,j_{q})\ |\ k_{j}=\delta^{j}_{i_{q+1}}+\cdots+\delta^{j}_{i_{p}}, i_{1}=j_{1},\cdots,i_{q}=j_{q} \}
\end{equation*}
is a solution to (\ref{eq4}). 

(3) Theorem \ref{th1} can be viewed as a generalization of some known results.

When $p=0,q=0$, Theorem \ref{th1} coincides with some results about first integrals in \cite{ref2, ref23}.

When $p=1,q=0$, Theorem \ref{th1} coincides with some results about Lie symmetries in \cite{xx16}.

(4) The resonant relation must be satisfied when $k_1=\cdots=k_n=0, ~p=q$ and $i_1=j_1, \cdots, i_p=j_p$.
\end{remark}

\section{Tensor invariants for semi-quasihomogeneous systems}

\subsection{Quasihomogeneous systems and semi-quasihomogeneous systems}
First of all  recall some notions for quasihomogeneous systems and semi-quasihomogeneous systems. For more details, see \cite{ref2}.

Consider a system of differential equations
\begin{equation}\label{eq12}
\dot{x}=g(x),\ \ x=(x^{1},\cdots,x^{n})\in\mathbb{C}^n.
\end{equation}

\begin{definition}
Let $s_{1},\cdots,s_{n}\in\mathbb{N}$. If for any $\rho\in \mathbb{R}^{+}$, all the components of $g=(g^{1},\cdots,g^{n})$ satisfy
\begin{equation}\label{eq13}
g^{j}(\rho^{s_{1}}x^{1},\cdots,\rho^{s_{n}}x^{n})
=\rho^{s_{j}+m-1}g^{j}(x^{1},\cdots,x^{n}),
\end{equation}
then system (\ref{eq12}) is called a quasi-homogeneous one of degree $m$ with exponents $s_1,\cdots,s_n$, where $m\in\mathbb{N}, m>1$.
\end{definition}

\begin{definition}
  Let $g(x)=g_{m}(x)+\tilde{g}(x)$. If $g_{m}(x)$ is a quasi-homogeneous vector field of degree $m$ with exponents $s_1,\cdots,s_n\in\mathbb{N}$, and $\tilde{g}(x)$ is the sum of quasi-homogeneous vector fields of degree all larger than $m$ (positively semi-quasihomogeneous) or all less than $m$ (negatively semi-quasihomogeneous), then system (\ref{eq12}) is called semi-quasihomogeneous.
\end{definition}

Let $S={\rm diag}(s_{1},\cdots,s_{n})$ be a diagonal matrix. We denote
${\rm diag}(\rho^{s_{1}},\cdots,\rho^{s_{n}})$ by $\rho^{S}$.

If system (\ref{eq12}) is semi-quasihomogeneous, then under the transformation
\begin{equation}\label{eq22}
x\mapsto\rho^{S}x, ~~t\mapsto\rho^{-\alpha}t, ~~\alpha=\frac{1}{m-1},
\end{equation}
it becomes
$$
\dot{x}=g_{m}(x)+\tilde{g}(x,\rho),
$$
where $\tilde{g}(x,\rho)$ is a formal power series either with respect to $\rho$ (positive semi-quasihomogeneity) or with respect to $\rho^{-1}$ (negative semi-quasihomogeneity) without any constant term.

We first consider the quasihomogeneous cut of semi-quasihomogeneous system (\ref{eq12})
\begin{equation}\label{eq23}
\dot{x}=g_{m}(x).
\end{equation}
System (\ref{eq23}) possesses particular solutions in the form
\begin{equation}\label{eq24}
x_{0}(t)=t^{-H}c=t^{\alpha S}c,
\end{equation}
where $H=\alpha S$ and the balance $c\neq0$ has to satisfy the algebraic system of equations $Hc+g_{m}(c)=0$.
Under the change of variables
\begin{equation}\label{eq26}
x=t^{-H}(c+u),~~ \tau=\ln t,
\end{equation}
system (\ref{eq23}) reads
\begin{equation}\label{KE}
u'=Ku+\tilde{f}(u),
\end{equation}
where $'=\frac{d}{d\tau}$, $K=H+\frac{\partial g_{m}}{\partial x}(c)$ is the so-called Kovalevskaya matrix. Its eigenvalues are called Kovalevskaya exponents. According to \cite{xx1}, $\lambda=-1$ is one of Kovalevskaya exponents.

\subsection{Necessary condition for tensor invariants of semi-quasihomogeneous systems}
\begin{definition}\label{def3}
A tensor field $T$ of type $(p,q)$ is called a quasi-homogeneous tensor field of degree $l$ with exponents $s_{1},\cdots,s_{n}\in\mathbb{N}$, if for any $\lambda\in{\mathbb{R}}^{+}$,
\begin{equation}\label{quasi}
T_{j_{1}\cdots j_{q}}^{i_{1}\cdots i_{p}}(\lambda^{s_{1}}x^{1},\cdots,\lambda^{s_{n}}x^{n}) =\lambda^{l-s_{j_{1}}-\cdots-s_{j_{q}}+s_{i_{1}}+\cdots+s_{i_{p}}}T_{j_{1}\cdots j_{q}}^{i_{1}\cdots i_{p}}(x).
\end{equation}
\end{definition}
\begin{remark}
For given $s_1,\cdots,s_n\in\mathbb{N}$, different tensor fields of type $(p,q)$ with same component function may be classified into quasi-homogeneous tensor fields of  different degrees. For example, for $(s_1, s_2)=(1, 2)$, $x^1\frac{\partial}{\partial x^1}\otimes dx^1$ is a quasi-homogeneous tensor field of degree $1$, while $x^1\frac{\partial}{\partial x^2}\otimes dx^1$ is a quasi-homogeneous tensor field of degree $0$.


\end{remark}

We will search for analytic (resp. polynomial) tensor invariants of positively (resp. negatively) semi-quasihomogeneous system (\ref{eq12}). Noticing that an analytic (resp. polynomial) tensor invariant $T(x)$ can be rescaled with the aid of the matrix $S$, i.e., it can be rewritten as follows
\begin{align}
 T(x)& =T_{l}(x)+T_{l+1}(x)+T_{l+2}(x)+\cdots\label{tensor}\\
({\rm resp.~~} T(x)& =T_{l}(x)+T_{l-1}(x)+T_{l-2}(x)+\cdots+T_{l-L}(x), L<l),\label{tensorr}
\end{align}
where $T_{r}(x)$ is a quasi-homogeneous tensor field of type $(p,q)$ of degree $r$ with
\begin{align}\label{Tr}
(T_{r})^{i_{1}\cdots i_{p}}_{j_{1}\cdots j_{q}}(x)=\sum_{s_{1}k_{1}+\cdots+s_{n}k_{n}
=r-s_{j_{1}}-\cdots-s_{j_{q}}+s_{i_{1}}+\cdots+s_{i_{p}}}a_{k_{1}\cdots k_{n}}(x^{1})^{k_{1}}\cdots(x^{n})^{k_{n}}.
\end{align}
\begin{remark}\label{remark3}
The degree of a quasi-homogeneous tensor field mentioned above may be positive, zero or negative. Specially, if $s_1=\cdots=s_n=1$, then for a quasi-homogeneous tensor field of type $(p,q)$, its degree $l$ satisfies $l\geq q-p$.
\end{remark}
\begin{lemma}\label{le3}
If system (\ref{eq12}) is positively semi-quasihomogeneous, then
$$
L_{g}T_{r}(\rho^S x)=\rho^{r+m-1}(L_{g_m}T_r(x)+O(\rho)).
$$
If system (\ref{eq12}) is negatively semi-quasihomogeneous, then
$$
L_{g}T_{r}(\rho^S x)=\rho^{r+m-1}(L_{g_m}T_r(x)+O(\rho^{-1})).
$$
\end{lemma}
\begin{proof}
If system (\ref{eq12}) is positively semi-quasihomogeneous, then for any $i_{r}, j_{s}\in\{1,\cdots,n\}$,
\begin{align}\label{eq2222}
 (L_{g}T_r)^{i_{1}\cdots i_{p}}_{j_{1}\cdots j_{q}}(x) = &g^{k}\frac{\partial}{\partial x^{k}}T^{i_{1}\cdots i_{p}}_{j_{1}\cdots j_{q}}(x)+T^{i_{1}\cdots i_{p}}_{kj_{2}\cdots j_{q}}\frac{\partial g^{k}}{\partial x^{j_{1}}}+\cdots+T^{i_{1}\cdots i_{p}}_{j_{1}\cdots j_{q-1}k}\frac{\partial g^{k}}{\partial x^{j_{q}}}
 \notag\\
 & -T^{li_{2}\cdots i_{p}}_{j_{1}\cdots j_{q}}\frac{\partial g^{i_{1}}}{\partial x^{l}}-\cdots-T^{i_{1}\cdots i_{p-1}l}_{j_{1}\cdots j_{q}}\frac{\partial g^{i_{p}}}{\partial x^{l}}.
 \end{align}
By (\ref{eq13}) and (\ref{quasi}), we have
\begin{align*}
\frac{\partial g_m^k}{\partial x^{j}}(\rho^S x)=&\rho^{-s_{j}+s_{k}+m-1}\frac{\partial g_m^k}{\partial x^{j}}(x),\\
\frac{\partial (T_r)_{j_{1}\cdots j_{q}}^{i_{1}\cdots i_{p}}}{\partial x^{j}}(\rho^S x) =&\rho^{-s_{j}+r-s_{j_{1}}-\cdots-s_{j_{q}}+s_{i_{1}}+\cdots+s_{i_{p}}}\frac{\partial (T_r)_{j_{1}\cdots j_{q}}^{i_{1}\cdots i_{p}}}{\partial x^{j}}(x).
\end{align*}
Therefore,
\begin{align}
&(L_{g}T_{r})^{i_{1}\cdots i_{p}}_{j_{1}\cdots j_{q}}(\rho^S x)\notag\\
=& \rho^{s_{k}+m-1}(g_{m}^{k}(x)+\tilde{g}^k(x,\rho))\rho^{-s_{k}+r-s_{j_{1}}-\cdots-s_{j_{q}}+s_{i_{1}}+\cdots+s_{i_{p}}}\frac{\partial}{\partial x^{k}}(T_{r})^{i_{1}\cdots i_{p}}_{j_{1}\cdots j_{q}}(x)\notag
\\& +\rho^{r-s_{k}-s_{j_{2}}-\cdots-s_{j_{q}}+s_{i_{1}}+\cdots+s_{i_{p}}}(T_{r})^{i_{1}\cdots i_{p}}_{kj_{2}\cdots j_{q}}(x)\notag
\rho^{s_{k}+m-1-s_{j_{1}}}\frac{\partial (g_{m}^{k}(x)+ \tilde{g}^{k}(x,\rho))}{\partial x^{j_{1}}}\notag
\\&+\cdots\notag
\\& +\rho^{r-s_{j_{1}}-\cdots-s_{j_{q-1}}-s_{k}+s_{i_{1}}+\cdots+s_{i_{p}}}(T_{r})^{i_{1}\cdots i_{p}}_{j_{1}\cdots j_{q-1}k}(x)\rho^{s_{k}+m-1-s_{j_{q}}}\frac{\partial (g_{m}^{k}(x)+ \tilde{g}^{k}(x,\rho))}{\partial x^{j_{q}}}\notag
\\& -\rho^{r-s_{j_{1}}-\cdots-s_{j_{q}}++s_{l}+s_{i_{2}}+\cdots+s_{i_{p}}}((T_{r})^{li_{2}\cdots i_{p}}_{j_{1}\cdots j_{q}}(x)\rho^{s_{i_{1}}+m-1-s_{l}}\frac{\partial (g_{m}^{i_{1}}(x)+\tilde{g}^{i_{1}}(x,\rho))}{\partial x^{l}}\notag
\\&-\cdots\notag
\\& -\rho^{r-s_{j_{1}}-\cdots-s_{j_{q}}++s_{1}+\cdots+s_{i_{p-1}}+s_{l}}(T_{r})^{i_{1}\cdots i_{p-1}l}_{j_{1}\cdots j_{q}}(x)\rho^{s_{i_{1}}+m-1-s_{l}}\frac{\partial (g_{m}^{i_{p}}(x)+\tilde{g}^{i_{p}}(x,\rho))}{\partial x^{l}}\notag
\\ =& \rho^{r+m-1-s_{j_{1}}-...-s_{j_{q}}+s_{i_{1}}+...+s_{i_{p}}}
((L_{g_m}T_r)^{i_{1}\cdots i_{p}}_{j_{1}\cdots j_{q}}(x)+O(\rho))\notag,
\end{align}
which leads to
\begin{align*}
&L_{g}T_{r}(\rho^{S}x)\notag
\\=& (L_{g}T_{r})^{i_{1}\cdots i_{p}}_{j_{1}\cdots j_{q}}(\rho^S x)\frac{\partial}{\partial (\rho^{s_{i_{1}}}x^{i_{1}})}\otimes\cdots\otimes\frac{\partial}{\partial (\rho^{s_{i_{p}}}x^{i_{p}})}\otimes d(\rho^{s_{j_{1}}}x^{j_{1}})\otimes\cdots\otimes d(\rho^{s_{j_{q}}}x^{j_{q}})
\\=& \rho^{r+m-1}(L_{g_m}T_r(x)+O(\rho)).
\end{align*}
The case of negatively semi-quasihomogeneous can be proved similarly.
\end{proof}
\begin{lemma}\label{lemma3}
Assume system (\ref{eq12}) is positively (resp. negatively) semi-quasihomogeneous. If $T(x)$ given by \eqref{tensor}(resp. \eqref{tensorr}) is an analytic (resp. polynomial) tensor invariant of (\ref{eq12}), then $T_{l}$ is a quasi-homogeneous tensor invariant of the truncated system (\ref{eq23}).
\end{lemma}
\begin{proof} Here we take the situation of positively semi-quasihomogeneous as an example, the other situation is similar to it.
According to Lemma \ref{le3},
$$
0= L_{g}T(\rho^{S}x)
= L_{g}(T_{l}(\rho^{S}x)+T_{l+1}(\rho^{S}x)+\cdots)
= \rho^{l+m-1}L_{g_m}T_l(x)+O(\rho^{l+m}).
$$
By the arbitrariness of $\rho$, we obtain
$$
L_{g_m}T_{l}=0,
$$
which means $T_{l}$ is a quasi-homogeneous tensor invariant of system (\ref{eq23}).
\end{proof}

Similarly, we can obtain the following conclusion.
\begin{lemma}\label{lemma4}
Assume system (\ref{eq12}) is quasi-homogeneous. If $T(x)$ is an analytic tensor invariant of (\ref{eq12}), then $T_{r}(r\geq l)$ is a quasi-homogeneous tensor invariant of system (\ref{eq12}).
\end{lemma}

Now we give a necessary condition for existence of tensor invariants in quasihomogeneous systems.

\begin{thm}\label{th2}
Assume system (\ref{eq12}) is semi-quasihomogeneous. If  (\ref{eq12}) possesses an analytic tensor invariant of type $(p,q)$ in a neighbourhood of the particular solution (\ref{eq24}), then the truncated system (\ref{eq23}) admits a quasi-homogeneous tensor invariant of degree $l$, and at least one of the following resonant condition holds:
\begin{equation}\label{eq032}
-\frac{l}{m-1} +\sum_{j=1}^{n}k_{j}\lambda_{j}=\lambda_{i_{1}}+\cdots+\lambda_{i_{p}}-\lambda_{j_{1}}-\cdots-\lambda_{j_{q}},
\ \ k_{j} \in \mathbb{N},
\end{equation}
where $\lambda_{1},\cdots,\lambda_{n}$ are Kovalevskaya exponents of (\ref{eq23}), $m$ is the degree of (\ref{eq23}), and $i_{r}, j_{s}\in\{1,\cdots,n\}$.
Moreover, if $l\geqslant0$, then \eqref{eq032} can be rewritten as
\begin{equation}\label{eq32}
\sum_{j=1}^{n}k_{j}\lambda_{j}=(m-1)(\lambda_{i_{1}}+\cdots+\lambda_{i_{p}}-\lambda_{j_{1}}-\cdots-\lambda_{j_{q}}),
~~k_{j} \in \mathbb{N},~~k_1\geqslant l.
\end{equation}
\end{thm}

\begin{proof} Let $T(x)$ be an analytic tensor invariant of type $(p,q)$ of system (\ref{eq12}), then by Lemma \ref{lemma3}, the truncated system (\ref{eq23}) admits a quasi-homogeneous tensor invariant $T_{l}(x)$.


Under the transformation (\ref{eq26}),
$T_{l}(x)$ becomes
\begin{align*}
&(T_{l})(t^{-H}(c+u))\\
=&~ \sum_{1\leq i_{1},\cdots,i_{p},j_{1},\cdots,j_{q}\leq n}(T_{l})_{j_{1}\cdots j_{q}}^{i_{1}\cdots i_{p}}(t^{-H}(c+u)) \frac{\partial}{\partial t^{-h_{i_{1}}}(c^{i_{1}}+u^{i_{1}})}\otimes\cdots\otimes\frac{\partial}{\partial t^{-h_{i_{p}}}(c^{i_{p}}+u^{i_{p}})}\notag\\
&~ \otimes d(t^{-h_{j_{1}}}(c^{j_{1}}+u^{j_{1}}))\otimes\cdots\otimes d(t^{-h_{j_{q}}}(c^{j_{q}}+u^{j_{q}}))\notag\\[1mm]
=&~ \sum_{1\leq i_{1},\cdots,i_{p},j_{1},\cdots,j_{q}\leq n}t^{-\alpha l+h_{j_{1}}+\cdots+h_{j_{q}}-h_{i_{1}}-\cdots-h_{i_{p}}}(T_{l})_{j_{1}\cdots j_{q}}^{i_{1}\cdots i_{p}}(c+u)(t^{h_{i_{1}}}\frac{\partial}{\partial u^{i_{1}}})\otimes\cdots \otimes(t^{h_{i_{p}}}\frac{\partial}{\partial u^{i_{p}}})\notag\\[1mm]
&~\otimes(t^{-h_{j_{1}}}du^{j_{1}}-h_{j_{1}}t^{-h_{j_{1}}-1}(c^{j_{1}}+u^{j_{1}})dt)\otimes\cdots \otimes(t^{-h_{j_{q}}}du^{j_{q}}-h_{j_{q}}t^{-h_{j_{q}}-1}(c^{j_{q}}+u^{j_{q}})dt)\notag\\[1mm]
=&~ t^{-\alpha l}\hat{T}(u)-t^{-\alpha l-1}\bar{T}(t,u),
\end{align*}
where
$$
\hat{T}(u)=\sum_{1\leq i_{1},\cdots,i_{p},j_{1},\cdots,j_{q}\leq n}(T_{l})_{j_{1}\cdots j_{q}}^{i_{1}\cdots i_{p}}(c+u)\frac{\partial}{\partial u^{i_{1}}}\otimes\cdots\otimes\frac{\partial}{\partial u^{i_{p}}}\otimes du^{j_{1}}\otimes\cdots\otimes du^{j_{q}},
$$
and $\bar{T}(t,u)$ denotes all tensor fields of type $(p,q)$ whose basis contains $t$.

Let $\tilde{T}(u^{0}, u)=(u^{0})^{\delta l}\hat{T}(u)$, then $\tilde{T}(u^{0}, u)$ is a tensor field of type $(p,q)$ with
\begin{align*}
& \tilde{T}(u^{0}, u)^{i_{1}\cdots i_{p}}_{j_{1}...j_{q}}=(u^{0})^{\delta l}(T_{l})_{j_{1}\cdots j_{q}}^{i_{1}\cdots i_{p}}(c+u),~~i_{r}, j_{s}\in\{1,\cdots,n\},
\\& \tilde{T}(u^{0}, u)^{i_{1}\cdots i_{p}}_{j_{1}...j_{q}}=0,~~i_{1}\cdots i_{p} j_{1}\cdots j_{q}=0.
\end{align*}
We claim that $\tilde{T}(u^{0},u)$
is an analytic tensor invariant of system
\begin{align}\label{new}
(u^{0})'=-\delta\alpha u^{0},~~u'=Ku+\tilde{f}(u),
\end{align}
where $u^{0}=t^{-\delta\alpha}=e^{-\delta\alpha\tau}$, $\delta={\rm sgn}(l)$, which means
\begin{align}\label{de}
(L_{\tilde{g}}\tilde{T}(u^0,u))^{i_{1}\cdots i_{p}}_{j_{1}...j_{q}}
=0
\end{align}
holds for any $i_{r}, j_{s}\in\{0,1,\cdots,n\}$, where $\tilde{g}$ denotes the vector field of (\ref{new}).

If $i_{1}\cdots i_{p} j_{1}\cdots j_{q}=0$, without loss of generality, we assume $i_{1}=0$, which means $\tilde{T}(u^0,u)^{0i_{2}...i_{p}}_{j_{1}...j_{q}}=0$, then by (\ref{eq2}), we have
\begin{equation}\nonumber
(L_{\tilde{g}}\tilde{T}(u^{0},u))^{0i_{2}\cdots i_{p}}_{j_{1}...j_{q}} =-\sum_{k=0}^{n}\tilde{T}(u^{0},u))^{ki_{2}\cdots i_{p}}_{j_{1}\cdots j_{q}}
\frac{\partial(-\delta\alpha u^{0})}{\partial u^{k}}=0.
\end{equation}

If $i_{r}, j_{s}\in\{1,\cdots,n\}$, then (\ref{de}) is equivalent to
\begin{align}
(L_{\tilde{g}}\tilde{T}(u^0,u))^{i_{1}\cdots i_{p}}_{j_{1}...j_{q}}
=&-\delta\alpha u^{0}\delta l(u^{0})^{\delta l-1}(T_{l})^{i_{1}\cdots i_{p}}_{j_{1}...j_{q}}(c+u)+(u^{0})^{\delta l}(L_{G}\hat{T}(u))^{i_{1}\cdots i_{p}}_{j_{1}...j_{q}}\notag
\\
=&(u^{0})^{\delta l}[-\alpha l(T_{l})^{i_{1}\cdots i_{p}}_{j_{1}...j_{q}}(c+u)+(L_{G}\hat{T}(u))^{i_{1}\cdots i_{p}}_{j_{1}...j_{q}}]=0,\notag
\end{align}
where $G(u)=Ku+f(u)$. By noticing that
$$
Ku+\tilde{f}(u)=g_{m}(c+u)+H(c+u)
$$
and linear property of Lie derivative,
\begin{align}\label{LG}
(L_{G}\hat{T}(u))^{i_{1}\cdots i_{p}}_{j_{1}...j_{q}}=(L_{g_m}\hat{T}(u))^{i_{1}\cdots i_{p}}_{j_{1}...j_{q}}+(L_{H}\hat{T}(u))^{i_{1}\cdots i_{p}}_{j_{1}...j_{q}}.
\end{align}

Since $T_{l}(x)$ is a quasi-homogeneous tensor invariant of system (\ref{eq23}), we have
\begin{align*}
g_{m}^{k}(x)\frac{\partial}{\partial x^{k}}(T_{l})^{i_{1}\cdots i_{p}}_{j_{1}\cdots j_{q}}(x)+(T_{l})^{i_{1}\cdots i_{p}}_{kj_{2}\cdots j_{q}}(x)\frac{\partial g_{m}^{k}(x)}{\partial x^{j_{1}}}+\cdots
+(T_{l})^{i_{1}\cdots i_{p}}_{j_{1}\cdots j_{q-1}k}(x)\frac{\partial g_{m}^{k}(x)}{\partial x^{j_{q}}}~&\notag
\\ -(T_{l})^{ki_{2}\cdots i_{p}}_{j_{1}...j_{q}}(x)\frac{\partial g^{i_{1}}(x)}{\partial x^{k}}
-\cdots-(T_{l})^{i_{1}\cdots i_{p-1}k}_{j_{1}\cdots j_{q}}(x)\frac{\partial g^{i_{p}}(x)}{\partial x^{k}}=&~0,
\end{align*}
therefore
\begin{align}\label{gm}
(L_{g_m}\hat{T}(u))^{i_{1}\cdots i_{p}}_{j_{1}...j_{q}}=&~g_{m}^{k}(c+u)\frac{\partial}{\partial u^{k}}(T_{l})^{i_{1}\cdots i_{p}}_{j_{1}\cdots j_{q}}(c+u)+(T_{l})^{i_{1}\cdots i_{p}}_{kj_{2}\cdots j_{q}}(c+u)\frac{\partial g_{m}^{k}(c+u)}{\partial u^{j_{1}}}+\cdots\notag\\
&~ +(T_{l})^{i_{1}\cdots i_{p}}_{j_{1}\cdots j_{q-1}k}(c+u)\frac{\partial g_{m}^{k}(c+u)}{\partial u^{j_{q}}}
-(T_{l})^{ki_{2}\cdots i_{p}}_{j_{1}\cdots j_{q}}(c+u)\frac{\partial g^{i_{1}}(c+u)}{\partial u^{k}}\notag\\
&~ -\cdots-(T_{l})^{i_{1}\cdots i_{p-1}k}_{j_{1}\cdots j_{q}}(c+u)\frac{\partial g^{i_{p}}(c+u)}{\partial u^{k}}\notag\\
=&~0.
\end{align}

Assume $H={\rm diag}(h_1,\cdots, h_n)$, then $h_i=\alpha s_i (i=1,2,\cdots,n)$,
\begin{align}\label{h1}
&(L_{H}\hat{T}(u))^{i_{1}\cdots i_{p}}_{j_{1}...j_{q}}\notag\\
=&(H(c+u))^{k}\frac{\partial (T_{l})^{i_{1}\cdots i_{p}}_{j_{1}\cdots j_{q}}(c+u)}{\partial u^{k}}\notag\\
& +(T_{l})^{i_{1}\cdots i_{p}}_{kj_{2}\cdots j_{q}}(c+u)\frac{\partial(H(c+u))^{k}}{\partial u^{j_{1}}}+\cdots+(T_{l})^{i_{1}\cdots i_{p}}_{j_{1}\cdots j_{q-1}k}(c+u)\frac{\partial(H(c+u))^{k}}{\partial u^{j_{q}}}\notag\\
& -(T_{l})^{li_{2}\cdots i_{p}}_{j_{1}\cdots j_{q}}(c+u)\frac{\partial(H(c+u))^{i_{1}}}{\partial u^{l}}-\cdots-(T_{l})^{i_{1}\cdots i_{p-1}l}_{j_{1}\cdots j_{q}}(c+u)\frac{\partial(H(c+u))^{i_{p}}}{\partial u^{l}}\notag\\
=& h_k(c^k+u^k)\frac{\partial (T_{l})^{i_{1}\cdots i_{p}}_{j_{1}\cdots j_{q}}(c+u)}{\partial u^{k}}\notag\\
& +\delta^k_{j_1}h_k(T_{l})^{i_{1}\cdots i_{p}}_{kj_{2}\cdots j_{q}}(c+u)+\cdots+\delta^k_{j_q}h_k(T_{l})^{i_{1}\cdots i_{p}}_{j_{1}\cdots j_{q-1}k}(c+u)\notag\\
& -\delta^{i_1}_{l}h_{i_{1}}(T_{l})^{li_{2}\cdots i_{p}}_{j_{1}\cdots j_{q}}(c+u)-\cdots--\delta^{i_p}_{l}h_{i_{p}}(T_{l})^{i_{1}\cdots i_{p-1}l}_{j_{1}\cdots j_{q}}(c+u)\notag\\
=& h_k(c^k+u^k)\frac{\partial (T_{l})^{i_{1}\cdots i_{p}}_{j_{1}\cdots j_{q}}(c+u)}{\partial u^{k}}\notag\\
&+(h_{j_{1}}+\cdots+h_{j_{q}}-h_{i_{1}}-\cdots-h_{i_{p}})(T_{l})^{i_{1}\cdots i_{p}}_{j_{1}\cdots j_{q}}(c+u).
\end{align}
By \eqref{Tr}, $(T_{l})^{i_{1}\cdots i_{p}}_{j_{1}\cdots j_{q}}(x)$ can be written as
$$
(T_{l})^{i_{1}\cdots i_{p}}_{j_{1}\cdots j_{q}}(x)=\sum_{s_{1}k_{1}+\cdots+s_{n}k_{n}
=l-s_{j_{1}}-\cdots-s_{j_{q}}+s_{i_{1}}+\cdots+s_{i_{p}}}a_{k_{1}\cdots k_{n}}(x^{1})^{k_{1}}\cdots(x^{n})^{k_{n}},
$$
thus
\begin{small}
\begin{align}\label{h2}
&\sum_{r=1}^{n}(h_r(c^r+u^r))\frac{\partial}{\partial u^{r}}(T_{l})^{i_{1}\cdots i_{p}}_{j_{1}\cdots j_{q}}(c+u)\notag\\[1mm]
=&\sum_{r=1}^{n}(h_r(c^r+u^r))\sum_{s_{1}k_{1}+\cdots+s_{n}k_{n}
=l-s_{j_{1}}-\cdots-s_{j_{q}}+s_{i_{1}}+\cdots+s_{i_{p}}}\frac{\partial}{\partial u^{r}}[a_{k_{1}\cdots k_{n}}(c^1+u^1)^{k_{1}}\cdots(c^n+u^n)^{k_{n}}]\notag\\[1mm]
=& \sum_{s_{1}k_{1}+\cdots+s_{n}k_{n}=l-s_{j_{1}}-\cdots-s_{j_{q}}+s_{i_{1}}+\cdots+s_{i_{p}}}a_{k_{1}\cdots k_{n}}
\sum_{r=1}^{n}(h_rk_r)(c^{1}+u^{1})^{k_{1}}\cdots(c^{n}+u^{n})^{k_{n}}
\notag\\[1mm]
=& \sum_{r=1}^{n}(\alpha s_rk_r)\sum_{s_{1}k_{1}+\cdots+s_{n}k_{n}=l-s_{j_{1}}-\cdots-s_{j_{q}}+s_{i_{1}}+\cdots+s_{i_{p}}}
a_{k_{1}\cdots k_{n}}(c^{1}+u^{1})^{k_{1}}\cdots(c^{n}+u^{n})^{k_{n}}\notag\\[1mm]
=& (\alpha l-h_{j_{1}}-\cdots-h_{j_{q}}+h_{i_{1}}+\cdots+h_{i_{p}})(T_{l})^{i_{1}\cdots i_{p}}_{j_{1}\cdots j_{q}}(c+u).
\end{align}
\end{small}

By (\ref{gm}), (\ref{h1}) and (\ref{h2}), $(L_{\tilde{g}}\tilde{T}(u^0,u))^{i_{1}\cdots i_{p}}_{j_{1}...j_{q}}=0$ holds for any $i_{r}, j_{s}\in\{1,\cdots,n\}$.

By Theorem \ref{th1}, at least one of resonant condition of (\ref{eq032}) type has to be fulfilled.

Moreover, since $\lambda_{1}=-1$, when $l\geqslant 0$, (\ref{eq032}) can be reduced to
\begin{equation}\label{42}
-l+(m-1)\sum_{j=1}^{n}k_{j}\lambda_{j}=(m-1)(\lambda_{i_{1}}+\cdots+\lambda_{i_{p}}-\lambda_{j_{1}}-\cdots-\lambda_{j_{q}}),
k_{j} \in\mathbb{N}.
\end{equation}
By rewrite $l+(m-1)k_{1}\rightarrow k_{1}$ and $(m-1)k_{j}\rightarrow k_{j}$, $j=2,\cdots,n$ in \eqref{42}, one can obtain (\ref{eq32}). The theorem is proved.
\end{proof}

\begin{remark}
Our result can be viewed as a generalization of Kozlov's result in \cite{xx1}. On the one hand, Kozlov's result presents the theorem on the existence of the quasi-homogeneous tensor invariants which are nonzero at $x=c$. our result is given without the restriction. On the other hand, our method is different from his. We can think of it as follows. Kozlov's result is considered on the particular solution. However, our result can be regarded as the first order variation along the paticular solution. More interestingly, by our method, we can even consider the problem of higher order variation.
\end{remark}

\section{Examples}
In this section, let us give several examples to illustrate our results.
\\
\noindent{\bf Example 1}
{\rm To illustrate Theorem \ref{th1}, we consider the following artificial system
\begin{align}\label{arti}
\begin{cases}
{\dot x_1}=x_1,\\
{\dot x_2}=\sqrt{2}x_2.
\end{cases}
\end{align}}

Obviously, $\lambda_1=1,\lambda_2=\sqrt{2}$ are the eigenvalues of the Jacobian matrix of system ({\ref{arti}}). By Theorem \ref{th1}, if system (\ref{arti}) has an analytic tensor invariant of type $(p,q)$, then at least one of the following resonant condition holds:
\begin{align}\label{reson0}
k_{1}+\sqrt{2}k_{2}=\lambda_{i_{1}}+\cdots+\lambda_{i_{p}}-\lambda_{j_1}-\cdots-\lambda_{j_q},
\end{align}
where $k_1, k_2\in\mathbb{N}, i_{r}, j_{s}\in\{1, 2\}$.
\begin{cor}
(1) Any analytic tensor invariant of type $(0,0)$ of degree $k$ of system (\ref{arti}) is trivial;

(2) Any analytic tensor invariant of type $(1,0)$ of system (\ref{arti}) has the form
\begin{align*}
T(x)=\alpha x_1\frac{\partial}{\partial x_1}+\beta x_2\frac{\partial}{\partial x_2};
\end{align*}

(3) Any analytic tensor invariant of type $(1,1)$ of system (\ref{arti}) has the form
\begin{align*}
T(x)=\alpha \frac{\partial}{\partial x_1}\otimes dx_1+\beta \frac{\partial}{\partial x_2}\otimes dx_2;
\end{align*}

(4) Any analytic tensor invariant of type $(2,0)$ of system (\ref{arti}) has the form
\begin{align*}
T(x)=\beta x_1^2 \frac{\partial}{\partial x_1}\otimes \frac{\partial}{\partial x_1}+\alpha x_1x_2\frac{\partial}{\partial x_1}\otimes \frac{\partial}{\partial x_2}+\gamma x_1x_2 \frac{\partial}{\partial x_2}\otimes \frac{\partial}{\partial x_1}+\zeta x_2^2 \frac{\partial}{\partial x_2}\otimes \frac{\partial}{\partial x_2};
\end{align*}

(5) System (\ref{arti}) admits no analytic tensor invariants of type $(p,q)$ when $q>p$;
where $\alpha,\beta,\gamma,\zeta$ are constants.
\end{cor}
\begin{proof}
Since system (\ref{arti}) is linear, by Lemma \ref{lemma1} and Corollary \ref{cor1}, if
\begin{align*}
T=T^{(k)}+T^{(k+1)}+\cdots
\end{align*}
is an analytic tensor field of type $(p,q)$ of (\ref{arti}), where $T^{(r)}(r=k,k+1,\cdots,)$ whose components are homogeneous polynomials of degree $r$, then $T^{(r)}(r=k,k+1,\cdots,)$ is an analytic tensor invariant of type $(p,q)$ of system (\ref{arti}) and (\ref{reson0}) hold, where $k=k_1+k_2$.

(1) Let $T=T(x_1,x_2)$ be an analytic tensor invariant of type $(0,0)$ of system (\ref{arti}), i.e.,
$T$ is an analytic first integral of system (\ref{arti}), then by Theorem C in \cite{ref23}, $T$ has to be a constant, i.e., $T$ is a trivial tensor invariant.

(2) By (\ref{reson0}), we have $k=k_1+k_2=1$.
Let
$$
T=(\beta_1x_1+\beta_2x_2)\frac{\partial}{\partial x_{1}}+(\gamma_1x_1+\gamma_2x_2)\frac{\partial}{\partial x_{2}}
$$
be a tensor invariant of type $(1,0)$ of system (\ref{arti}), where $\beta_i, \gamma_j$ are constants. Let $F$ denote the vector field of (\ref{arti}), i.e.,
$$
F=x_1\frac{\partial}{\partial x_{1}}+\sqrt{2}x_2\frac{\partial}{\partial x_{2}}.
$$ By $(L_{F}T)^i=0, i=1, 2$ we get
\begin{align*}
(\sqrt{2}-1)\beta_2x_2=&~0,\\
(\sqrt{2}-1)\gamma_1x_1=&~0,
\end{align*}
which leads to $\beta_2=\gamma_1=0$. Therefore
\begin{align*}
T(x)=\alpha x_1\frac{\partial}{\partial x_1}+\beta x_2\frac{\partial}{\partial x_2}.
\end{align*}

(3) By (\ref{reson0}), we have $k=k_1+k_2=0$.
Let
$$
T=\sum_{i,j=1}^3\beta_{ij}\frac{\partial}{\partial x_i}\otimes dx_j
$$
be a tensor invariant of type $(1,1)$ of system (\ref{arti}), where $\beta_i, \gamma_j$ are constants.
By $(L_{F}T)^i_j=0, i,j=1, 2$ we get
\begin{align*}
(\sqrt{2}-1)\beta_{12}=&~0,\\
(\sqrt{2}-1)\beta_{21}=&~0,
\end{align*}
which lead to $\beta_{ij}=0$ for $i\neq j$. Therefore
\begin{align*}
T=\beta_{11} \frac{\partial}{\partial x_1}\otimes dx_1+\beta_{22} \frac{\partial}{\partial x_2}\otimes dx_2.
\end{align*}

(4) By (\ref{reson0}), we have $k=k_1+k_2=2$.
Let
$$
T=\sum_{i,j=1}^3(\beta_{ij}x_1^2+\alpha_{ij}x_1x_2+\gamma_{ij}x_2^2)\frac{\partial}{\partial x_i}\otimes \frac{\partial}{\partial x_j}
$$
be a tensor invariant of type $(2,0)$ of system (\ref{arti}), where $\beta_{ij},\alpha_{ij}, \gamma_{ij}$ are constants.
By $(L_{F}T)^{ij}=0, i,j=1, 2$ we get
\begin{align*}
(\sqrt{2}-1)\alpha_{11}x_1x_2+(2\sqrt{2}-2)\gamma_{11}x_2^2=&~0,\\
(\sqrt{2}-1)\beta_{12}x_1^2+(1-\sqrt{2})\gamma_{12}x_2^2=&~0,\\
(\sqrt{2}-1)\beta_{21}x_1^2+(1-\sqrt{2})\gamma_{21}x_2^2=&~0,\\
(2-2\sqrt{2})\beta_{22}x_1^2+(1-\sqrt{2})\alpha_{22}x_1x_2=&~0,
\end{align*}
which lead to $\alpha_{11}=\gamma_{11}=\beta_{12}=\gamma_{12}=\beta_{21}=\gamma_{21}
=\alpha_{22}=\beta_{22}=0$. Therefore
\begin{align*}
T(x)=\beta x_1^2 \frac{\partial}{\partial x_1}\otimes \frac{\partial}{\partial x_1}+\alpha x_1x_2\frac{\partial}{\partial x_1}\otimes \frac{\partial}{\partial x_2}+\gamma x_1x_2 \frac{\partial}{\partial x_2}\otimes \frac{\partial}{\partial x_1}+\zeta x_2^2 \frac{\partial}{\partial x_2}\otimes \frac{\partial}{\partial x_2}.
\end{align*}

(5) By (\ref{reson0}) and $q>p$, one can obtain $0\leqslant k =p-q<0$, a contraction.
\end{proof}

\noindent{\bf Example 2}
{\rm Consider a three-dimensional Lotka-Volterra system
\begin{align}\label{ex1}
\begin{cases}
\dot{x}_{1}=2x_{1}^{2}+2x_{1}x_{2}+x_{1}x_{3},
\\\dot{x}_{2}=-x_{1}x_{2}+x_{2}^{2}+3x_{2}x_{3},
\\\dot{x}_{3}=3x_{2}x_{3}-x_{3}^{2}.
\end{cases}
\end{align}}

System (\ref{ex1}) is a quasi-homogeneous system of degree 2 with exponents $s_{1}=s_{2}=s_{3}=1$. So $m=2, S=E={\rm diag}(1,1,1), \alpha=\frac{1}{m-1}=1, H=\alpha S=E$.
Since the following algebraic equations
\begin{equation}\nonumber
\begin{aligned}
\begin{cases}
c_{1}+2c_{1}^{2}+2c_{1}c_{2}+c_{1}c_{3}=0,
\\c_{2}-c_{1}c_{2}+c_{2}^{2}+3c_{2}c_{3}=0,
\\c_{3}+3c_{2}c_{3}-c_{3}^{2}=0
\end{cases}
\end{aligned}
\end{equation}
have a nonzero solution $c=(0,-1,0)$, system (\ref{ex1}) has a particular solution $x_{0}(t)=t^{-H}c=t^{E}c$.
The corresponding Kovalevskaya matrix
\begin{equation}
K=\left(\begin{array}{ccc}
-1 & 0  & 0 \\
1  & -1 & -3 \\
0  & 0  & -2 \\
\end{array}
\right)
\end{equation}
with Kovalevskaya exponents $\lambda_{1}=-1,\lambda_{2}=-1,\lambda_{3}=-2$.
\begin{cor}
(1) Any analytic tensor invariant of type $(0,0)$ of system (\ref{ex1}) is trivial;

(2) Any analytic tensor invariant of type $(1,1)$ of system (\ref{ex1}) is trivial;



(3) Any analytic tensor invariant of type $(1,0)$ of system (\ref{ex1}) is the vector field itself up to a constant;

(4) System (\ref{ex1}) admits no analytic tensor invariants of type $(p,q)$ when $q>\frac{3}{2}p$;
\end{cor}
\begin{proof}
Since system (\ref{ex1}) is quasi-homogeneous, by Lemma \ref{lemma4} and Theorem \ref{th2}, if
\begin{align*}
 T(x)& =T_{l}(x)+T_{l+1}(x)+T_{l+2}(x)+\cdots
\end{align*}
is an analytic tensor invariant of type $(p,q)$ of (\ref{ex1}), where
$T_{r}(x)(r\geq l)$ is a quasi-homogeneous tensor field of type $(p,q)$ of degree $r$
with
\begin{align}\label{Tre}
(T_{r})^{i_{1}\cdots i_{p}}_{j_{1}\cdots j_{q}}(x)=\sum_{s_{1}k_{1}+\cdots+s_{n}k_{n}
=r-s_{j_{1}}-\cdots-s_{j_{q}}+s_{i_{1}}+\cdots+s_{i_{p}}}a_{k_{1}\cdots k_{n}}(x^{1})^{k_{1}}\cdots(x^{n})^{k_{n}},
\end{align}
then $T_{r}(r\geq l)$ is a quasi-homogeneous tensor invariant of degree $r$ of system (\ref{ex1}), and at least one of the following resonant condition holds:
\begin{align*}
-r-k_{1}-k_{2}-2k_{3}=\lambda_{i_{1}}+\cdots+\lambda_{i_{p}}-\lambda_{j_1}-\cdots-\lambda_{j_q},
\end{align*}
where $k_1, k_2, k_3\in \mathbb{N},~~i_{r}, j_{s}\in\{1,2,3\}$. Together with Remark \ref{remark3}, we have
\begin{align}\label{exre1}
q-p\leq r=-(k_{1}+k_{2}+2k_{3}+\lambda_{i}+\cdots+\lambda_{i_{p}}-\lambda_{j_1}\cdots-\lambda_{j_q})
\le 2p-q.
\end{align}


(1) Let $T$ be of type $(0,0)$, by \eqref{exre1}, $r=0$, by \eqref{Tre}, $T$ has to be a constant, i.e., $T$ is a trivial tensor invariant.

(2) Let $T$ be of type $(1,1)$, by (\ref{exre1}), $0\le r\le 1$.

If $r=0$, by \eqref{Tre}, then $T_r$ has the form
$$
T_r=\sum_{i,j=1}^3\beta_{ij}\frac{\partial}{\partial x_i}\otimes dx_j,
$$
where $\beta_{ij}$ are constants. Denote by $F$ the vector field associated to the system (\ref{ex1}), i.e.,
$$
F=(2 x_1^2+2x_1x_2+x_1x_3)\frac{\partial}{\partial x_{1}}+(-x_1x_2+x_2^2+3x_2x_3)\frac{\partial}{\partial x_{2}}+(3x_2x_3-x_3^2)\frac{\partial}{\partial x_{3}},
$$
then $(L_{F}T_r)^i_j=0(i,j=1,2,3)$ imply that
\begin{align*}
 \beta_{12}x_2+(2\beta_{21}+\beta_{31})x_1=&~0,\\
 (2\beta_{11}-2\beta_{22}-\beta_{32})x_1+3\beta_{13}x_3=&~0,\\
 (\beta_{11}-2\beta_{23}-\beta_{33})x_1=&~0,\\ 5\beta_{21}x_1+(-\beta_{22}+\beta_{11}-3\beta_{31})x_2-2\beta_{21}x_3=&~0,\\
 3\beta_{23}x_3-3\beta_{32}x_2=&~0,\\
 (3\beta_{22}-3\beta_{33})x_2=&~0,\\
 4\beta_{31}x_1-\beta_{31}x_2+3\beta_{31}x_{3}=&~0,\\
 (3\beta_{33}-3\beta_{22})x_3=&~0,
\end{align*}
which lead to $\beta_{ij}=0$ for $i\neq j$ and $\beta_{11}=\beta_{22}=\beta_{33}$. Therefore
\begin{equation*}
T_r=\beta_{11}\frac{\partial}{\partial x_{1}}\otimes dx_{1}+\beta_{11}\frac{\partial}{\partial x_{2}}\otimes dx_{2}+\beta_{11}\frac{\partial}{\partial x_{3}}\otimes dx_{3}.
\end{equation*}

If $r=1$, by \eqref{Tre}, then $T_r$ has the form
$$
T_r=\sum_{i,j=1}^3(\beta_{ij}x_1+\gamma_{ij}x_2+\zeta_{ij}x_3)\frac{\partial}{\partial x_i}\otimes dx_j,
$$
where $\beta_{ij}, \gamma_{ij}, \zeta_{ij}$ are constants. Similarly, by a series of complex calculations, we obtain $\beta_{ij}=\gamma_{ij}=\zeta_{ij}=0$.

In conclusion, any analytic tensor invariant of type $(1,1)$ of system (\ref{ex1}) is trivial.

(3) Let $T$ be of type $(1,0)$, by \eqref{exre1}, $-1\leq r\leq 2$.

{\bf Case $1$:} If $r=-1$, by \eqref{Tre}, then $T_r$ has the form
$$
T_r=\beta\frac{\partial}{\partial x_{1}}+\gamma\frac{\partial}{\partial x_{2}}+\zeta\frac{\partial}{\partial x_{3}},
$$
where $\beta, \gamma, \zeta$ are constants. By $(L_{F}T_r)^i=0(i=1, 2, 3)$,  we get
\begin{align*}
(-4a\beta-2\gamma-\zeta)x_1-2\beta x_2-\beta x_3=&~0,\\
\gamma x_1+(\beta-2\gamma-3\zeta)x_2=&~0,\\
(2\zeta-3\gamma)x_3=&~0,
\end{align*}
which leads to $\beta=\gamma=\zeta=0$.

{\bf Case $2$:} If $r=0$, by \eqref{Tre}, then $T_r$ has the form
$$
T_r=(\beta_1x_1+\beta_2x_2+\beta_3x_3)\frac{\partial}{\partial x_{1}}+(\gamma_1x_1+\gamma_2x_2+\gamma_3x_3)\frac{\partial}{\partial x_{2}}+(\zeta_1x_1+\zeta_2x_2+\zeta_3x_3)\frac{\partial}{\partial x_{3}},
$$
where $\beta_{i}, \gamma_i, \zeta_i$ are constants.
Similarly, we can obtain that $\beta_i=\gamma_j=\zeta_r=0$.

{\bf Case $3$:} If $r=1$, by \eqref{Tre}, then $T_r$ has the form
$$
T_r=(\sum_{i,j=1,i\leq j}^3\beta_{ij}x_ix_j)\frac{\partial}{\partial x_{1}}+(\sum_{i,j=1,i\leq j}^3\gamma_{ij}x_ix_j)\frac{\partial}{\partial x_{2}}+(\sum_{i,j=1,i\leq j}^3\zeta_{ij}x_ix_j)\frac{\partial}{\partial x_{3}},
$$
where $\beta_{ij}, \gamma_{ij}, \zeta_{ij}$ are constants. Similarly, we derive
\begin{align*}
&\beta_{22}=\beta_{23}=\beta_{33}=\gamma_{11}=\gamma_{13}=\gamma_{33}=\zeta_{11}=\zeta_{12}
=\zeta_{13}=\zeta_{22}=0,\\
&\beta_{11}=\beta_{12}=2\beta_{13}=-2\gamma_{12}=2\gamma_{22}=\frac{2}{3}\gamma_{23}=
\frac{2}{3}\zeta_{23}=-2\zeta_{33},
\end{align*}
therefore
$$
T_r=\frac{1}{2}\beta_{11}[(2 x_1^2+2x_1x_2+x_1x_3)\frac{\partial}{\partial x_{1}}+(-x_1x_2+x_2^2+3x_2x_3)\frac{\partial}{\partial x_{2}}+(3x_2x_3-x_3^2)\frac{\partial}{\partial x_{3}}]=\frac{1}{2}\beta_{11} F.
$$

{\bf Case $4$:} If $r=2$, by \eqref{Tre}, then $T_r$ has the form
$$
T_r=(\sum_{i+j+r=3}\beta_{ijr}x_1^ix_2^jx_3^r)\frac{\partial}{\partial x_{1}}+(\sum_{i+j+r=3}\gamma_{ijr}x_1^ix_2^jx_3^r)\frac{\partial}{\partial x_{2}}+(\sum_{i+j+r=3}\zeta_{ijr}x_1^ix_2^jx_3^r)\frac{\partial}{\partial x_{3}},
$$
where $\beta_{ijr}, \gamma_{ijr}, \zeta_{ijr}$ are constants. After some cumbersome calculations, we obtain $\beta_{ijr}=\gamma_{ijr}=\zeta_{ijr}=0$.

In conclusion, any analytic tensor invariant of type $(1,0)$ of system (\ref{ex1}) is the vector field itself up to a constant.

(4) By $q> \frac{3}{2}p$ and \eqref{exre1}, one can obtain $\frac{p}{2}<q-p\leq l\le 2p-q<\frac{p}{2}$, a contradiction.
\end{proof}

\noindent{\bf Example 3}
{\rm Consider a perturbed oregonator model \cite{ref2,xx6}
\begin{align}\label{ex21}
\begin{cases}
\dot{x}=\alpha(y-xy+x-\varepsilon xz-gx),
\\\dot{y}=\alpha^{-1}(-y-xy+fz),
\\\dot{z}=\beta(x-\varepsilon xz-z).
\end{cases}
\end{align}}

This system of equations describes a hypothetical chemical reaction of the Belousov-Zhabotinsky type where variables $x, y, z$ mean concentrations of reagents. From the model, $\alpha, \beta, \varepsilon, f, g>0$. We notice that system (\ref{ex21}) is negatively semi-quasihomogeneous. Its cut can be written as follows
\begin{align}\label{ex2}
\begin{cases}
\dot{x}=-\alpha x(y+gx),
\\\dot{y}=-\alpha^{-1}xy,
\\\dot{z}=\beta x(1-\varepsilon z).
\end{cases}
\end{align}
System (\ref{ex2}) is a quasi-homogeneous system of degree 2 with exponents $s_{1}=s_{2}=1,s_{3}=0$. So $m=2, \alpha=\frac{1}{m-1}=1, S=H=diag(1,1,0)$.
Since the following algebraic equations
\begin{equation}\nonumber
\begin{aligned}
\begin{cases}
c_{1}-\alpha c_{1}c_{2}-\alpha gc_{1}^{2}=0,
\\c_{2}-\alpha^{-1}c_{1}c_{2}=0,
\\\beta c_{1}-\beta\varepsilon c_{1}c_{3}=0
\end{cases}
\end{aligned}
\end{equation}
have a nonzero solution $c=(\alpha,\alpha^{-1}-\alpha g,\varepsilon^{-1})$, system (\ref{ex2}) has a particular solution $x_{0}(t)=t^{-H}c$. The corresponding Kovalevskaya matrix
\begin{equation}
{\left(\begin{array}{ccc}
-g\alpha^{2}  &  -\alpha^{2}  & 0 \\
-\alpha^{2}+g  & 0 & 0 \\
0  & 0  & -\alpha\beta\varepsilon \\
\end{array}
\right)}
\end{equation}
with Kovalevskaya exponents $\lambda_{1}=-1,\lambda_{2}=1-g\alpha^{2},\lambda_{3}=-\alpha\beta\varepsilon$.
\begin{cor}
(1) If $g\alpha^{2}> 1$, then any analytic tensor invariant of type $(0,0)$ of system (\ref{ex21}) is trivial;

(2) If $2< g\alpha^{2}< 3, 1< \alpha\beta\varepsilon < 2$, then any analytic tensor invariant of type $(1,1)$ of system (\ref{ex21}) is trivial;

(3) If $2< g\alpha^{2}< 3, 1< \alpha\beta\varepsilon < 2$, then system (\ref{ex21}) admits no analytic tensor invariants of type $(1,q)$ when $q\geqslant 2$;

(4) If $2< g\alpha^{2}< 3, 1< \alpha\beta\varepsilon < 2$, then system (\ref{ex21}) admits no analytic tensor invariants of type $(p,q)$ when $q\geqslant 3p$. Moreover, system (\ref{ex21}) admits no analytic tensor invariants  of degree $l\geqslant 0$ of type $(p,q)$ when $q\geqslant 2p$.

(5) If $1< g\alpha^{2}< 3, \alpha\beta\varepsilon < 2, \frac{\alpha g+\alpha^{-1}}{\beta\varepsilon}\notin\mathbb{Z}, \frac{2\alpha g+\alpha^{-1}}{\beta\varepsilon}\notin\mathbb{Z}$, then system (\ref{ex2}) admits only a tensor invariant of type $(1,0)$ except for the vector field itself which is $T=(\varepsilon z-1)\frac{\partial}{\partial z}$ .
\end{cor}

\begin{proof} According to Lemma \ref{lemma3} and Theorem \ref{th2}, if system (\ref{ex21}) has an analytic tensor invariant of type $(p,q)$, then system (\ref{ex2}) admits a quasi-homogeneous tensor invariant of degree $l$, and at least one of the following resonant condition holds:
\begin{align*}
-l-k_{1}+(1-g\alpha^{2})k_{2}-\alpha\beta\varepsilon k_{3}=\lambda_{i_{1}}+\cdots+\lambda_{i_{p}}-\lambda_{j_1}-\cdots-\lambda_{j_q},
\end{align*}
where $k_1, k_2, k_3\in \mathbb{N},~~i_{r}, j_{s}\in\{1,\cdots,n\}$, which together with Remark \ref{remark3} lead to
\begin{align}
-p\leq l=-(k_{1}+k_{2}+2k_{3}+\lambda_{i}+\cdots+\lambda_{i_{p}}-\lambda_{j_1}\cdots-\lambda_{j_q}).
\end{align}
(1) If system (\ref{ex2}) has a quasi-homogeneous tensor invariant $T$ of type $(0,q)$ of degree $l$, then at least one of the following resonant condition holds:
$$
-l-k_{1}+(1-g\alpha^{2})k_{2}-\alpha\beta\varepsilon k_{3}=\lambda_{i_{1}}+\cdots+\lambda_{i_{p}}-\lambda_{j_{1}}-\cdots-\lambda_{j_{q}},
~~ k_1, k_2, k_3 \in \mathbb{N}.
$$
If $g\alpha^{2}> 1$, then
$$
l=\lambda_{j_{1}}+\cdots+\lambda_{j_{q}}-k_{1}+(1-g\alpha^{2})k_{2}-\alpha\beta\varepsilon k_{3}\le 0.
$$
Meanwhile, Similar to Remark \ref{remark3}, $l\geq 0$. Next, we discuss the case of $l=0$.

When $l=0$, which implies that $q=0$, $T$ has the form
$$
T=f(z),
$$
where $f$ is an analytic function. Let $G$ denote the vector field of (\ref{ex2}). By (\ref{eq2}), we have
$$
L_GT=\beta x(1-\varepsilon z)f'(z)=0,
$$
which leads to $f(z)=const$. It is a trivial tensor invariant of type $(0,0)$.

Except the above case, $l>0$. This is a contradiction, which means that system (\ref{ex2}) admits no nontrivial analytic tensor invariants of type $(0,q)$.

(2) If system (\ref{ex2}) has a quasi-homogeneous tensor invariant $T$ of type $(1,1)$ of degree $l$, then at least one of the following resonant condition holds:
$$
-l-k_{1}+(1-g\alpha^{2})k_{2}-\alpha\beta\varepsilon k_{3}=\lambda_{i}-\lambda_{j},~~ k_1, k_2, k_3 \in \mathbb{N}.
$$
If $2< g\alpha^{2}< 3, 1< \alpha\beta\varepsilon < 2$, then
$$
l=\lambda_{j}-\lambda_{i}-k_{1}+(1-g\alpha^{2})k_{2}-\alpha\beta\varepsilon k_{3}< \mu,
$$
where $0<\mu< 1$. Meanwhile, Similar to Remark \ref{remark3}, $l\geq -1$.

Case $1$: $T$ has the form
$$
T=f(z)\frac{\partial}{\partial x}\otimes dz+h(z)\frac{\partial}{\partial y}\otimes dz,
$$
By (\ref{eq2}), $(L_{G}T)^1_1=(L_{G}T)^2_1=0$, i.e.,
\begin{align*}
 f(z)\beta(1-\varepsilon z)=&~0,\\
 h(z)\beta(1-\varepsilon z)=&~0,
\end{align*}
which leads to $f(z)=h(z)=0$.

Case $2$: $T$ has the form
\begin{align*}
T=& f_{11}(z)\frac{\partial}{\partial x}\otimes dx+f_{12}(z)\frac{\partial}{\partial x}\otimes dy+[xf_{13}(z)+yh_{13}(z)]\frac{\partial}{\partial x}\otimes dz+f_{21}(z)\frac{\partial}{\partial y}\otimes dx \\
& +f_{22}(z)\frac{\partial}{\partial y}\otimes dy+[xf_{23}(z)+yh_{23}(z)]\frac{\partial}{\partial y}\otimes dz +f_{33}(z)\frac{\partial}{\partial z}\otimes dz.
\end{align*}
By a series of cumbersome calculations, we obtain $f_{11}=f_{22}=f_{33}=const., ~f_{ij}=h_{ij}=0$, which is same to Example $1$.

Consequently, any analytic tensor invariant of type $(1,1)$ of system (\ref{ex2}) is trivial.

(3) When $q\geq 2$, if system (\ref{ex2}) has an analytic tensor invariant of type $(1,q)$, then at least one of the following resonant condition holds:
$$
-l-k_{1}+(1-g\alpha^{2})k_{2}-\alpha\beta\varepsilon k_{3}=\lambda_{i}-\lambda_{j_1}-\cdots-\lambda_{j_q},~~ k_1, k_2, k_3 \in \mathbb{N}.
$$
Obviously,
$$
l=-k_{1}+(1-g\alpha^{2})k_{2}-\alpha\beta\varepsilon k_{3}-\lambda_{i}+\lambda_{j_1}\cdots+\lambda_{j_q}<\nu,
$$
where $\nu< 0$.
When $l=-1$, by a short calculation, we obtain $T=0$.
Consequently, system (\ref{ex2}) has no analytic tensor invariants of type $(1,q)$.

(3) If system (\ref{ex2}) has an analytic tensor invariant of type $(1,q)$, then at least one of the following resonant condition holds:
$$
-l-k_{1}+(1-g\alpha^{2})k_{2}-\alpha\beta\varepsilon k_{3}=\lambda_{i_1}+\cdots+\lambda_{i_p}-\lambda_{j_1}-\cdots-\lambda_{j_q},~~ k_1, k_2, k_3 \in \mathbb{N}.
$$
Obviously,
$$
l=-k_{1}+(1-g\alpha^{2})k_{2}-\alpha\beta\varepsilon k_{3}-\lambda_{i_1}-\cdots-\lambda_{i_p}+\lambda_{j_1}\cdots+\lambda_{j_q}<\nu,
$$
where $\nu< 2p-q$. If $2< g\alpha^{2}< 3, 1< \alpha\beta\varepsilon < 2$, $2p-q\leq-p$. Consequently, system (\ref{ex2}) has no analytic tensor invariants of type $(p,q)$ when $p\geqslant2, q\geqslant 3p$. Moreover, if $l\geq0$, then system (\ref{ex2}) has no analytic tensor invariants of type $(1,q)$ of degree $l\geqslant 0$ when $p\geqslant2, q\geqslant 2p$.

(4) If system (\ref{ex2}) has a quasi-homogeneous tensor invariant $T$ of type $(1,0)$ of degree $l$, then at least one of the following resonant condition holds:
$$
-l-k_{1}+(1-g\alpha^{2})k_{2}-\alpha\beta\varepsilon k_{3}=\lambda_{i},~~k_1, k_2, k_3\in\mathbb{N}.
$$
Obviously,
$$
l=-(k_{1}+(g\alpha^{2}-1)k_{2}+\alpha\beta\varepsilon k_{3}+\lambda_{i})\le \mu,
$$
where $0<\mu<2$. By Remark \ref{remark3}, we have $l\geq -1$.

Case $1$: $T$ has the form
$$
T=f(z)\frac{\partial}{\partial x}+h(z)\frac{\partial}{\partial y},
$$
By (\ref{eq2}), $(L_{G}T)^1=0$, i.e.,
$$
\beta x(1-\varepsilon z)f'(z)-(-\alpha y-2\alpha gx)f(z)+\alpha xh(z)=0,
$$
which leads to $f(z)=h(z)=0$.

Case $2$: $T$ has the form
$$
T=(xf_1(z)+yh_1(z))\frac{\partial}{\partial x_{1}}+(xf_2(z)+yh_2(z))\frac{\partial}{\partial x_{2}}+f_3(z)\frac{\partial}{\partial x_{3}},
$$
By (\ref{eq2}), $(L_{G}T)^i=0, i=1, 2, 3$, i.e.,
\begin{align*}
(-\alpha h_1(z)f_1(z)+\beta f'_1(z)-\beta\varepsilon zf_1'(z)+2\alpha h_1(z)+\alpha f_2(z))x^2+\alpha h_2(z)xy+\alpha h_1(z)y^2=&~0,\\
(-\alpha h_1(z)f_2(z)+\beta f'_2(z)-\beta\varepsilon zf_2'(z)-\alpha^{-1}f_2(z))x^2-(\alpha f_2(z)-\alpha^{-1}f_1(z))xy=&~0,\\
(\beta f'_3(z)-\beta\varepsilon zf_3'(z)-\beta f_1(z)+\beta\varepsilon f_3(z))x=&~0.
\end{align*}
If $\frac{\alpha g+\alpha^{-1}}{\beta\varepsilon}\notin\mathbb{Z}, \frac{2\alpha g+\alpha^{-1}}{\beta\varepsilon}\notin\mathbb{Z}$, we can obtain $T=(\varepsilon z-1)\frac{\partial}{\partial z}$.

Case $3$: $T$ has the form
\begin{align*}
T=& [(x^2f_{1}(z)+xyg_1(z)+y^2h_1(z))f(z)]\frac{\partial}{\partial x_{1}}+[(x^2f_{2}(z)+xyg_2(z)+y^2h_2(z))f(z)]\frac{\partial}{\partial x_{2}}\\
& +(xf_{3}(z)+yh_{3}(z))\frac{\partial}{\partial x_{3}}.
\end{align*}
Similarly, by a series of cumbersome calculations, we found that in this case $T$ is the vector field itself of system (\ref{ex2}).
Consequently, system (\ref{ex2}) admit only a tensor invariant of type $(1,0)$ except for the vector field itself which is $X=(\varepsilon z-1)\frac{\partial}{\partial z}$.
\end{proof}

\end{document}